\documentclass{amsart}
\usepackage[latin1]{inputenc}
\usepackage{amssymb}
\usepackage{latexsym}
\usepackage{amscd}
\usepackage{longtable}
\usepackage{amsmath}
\usepackage{mathrsfs}
\usepackage{array}
\usepackage[all]{xy}
\usepackage{a4}

\newtheorem{defn0}{Definition}[section]
\newtheorem{prop0}[defn0]{Proposition}
\newtheorem{thm0}[defn0]{Theorem}
\newtheorem{lemma0}[defn0]{Lemma}
\newtheorem{claim0}[defn0]{Claim}
\newtheorem{corollary0}[defn0]{Corollary}
\newtheorem{example0}[defn0]{Example}
\newtheorem{remark0}[defn0]{Remark}
\newtheorem{assumption0}[defn0]{Assumption}
\newtheorem{conjecture0}[defn0]{Conjecture}
\newtheorem{notation0}[defn0]{Notation}
\newtheorem{question0}[defn0]{Question}

\newenvironment{definition}{\begin{defn0}\rm}{\end{defn0}}
\newenvironment{proposition}{\begin{prop0}}{\end{prop0}}
\newenvironment{theorem}{\begin{thm0}}{\end{thm0}}
\newenvironment{lemma}{\begin{lemma0}}{\end{lemma0}}

\newenvironment{corollary}{\begin{corollary0}}{\end{corollary0}}

\newenvironment{remark}{\begin{remark0}\rm}{\end{remark0}}
\newenvironment{assumption}{\begin{assumption0}\rm}{\end{assumption0}}

\newcommand{\M}{\mathrm{M}}

\newcommand{\GL}{{\mathrm{GL}}}
\newcommand{\End}{{\mathrm{End}}}

\newcommand{\Z}{{\mathbb Z}}
\newcommand{\A}{{\mathbb A}}

\newcommand{\Q}{{\mathbb Q}}
\newcommand{\C}{{\mathbb C}}
\newcommand{\R}{{\mathbb R}}

\newcommand{\N}{{\mathbb N}}
\newcommand{\G}{{\mathbb G}}

\newcommand{\PP}{{\mathbb P}}

\newcommand{\cR}{{\mathcal R}}

\newcommand{\cC}{{\mathcal C}}

\newcommand{\cH}{{\mathcal H}}

\newcommand{\cL}{{\mathcal L}}
\newcommand{\cM}{{\mathcal M}}

\newcommand{\cP}{{\mathcal P}}
\newcommand{\cO}{{\mathcal O}}
\newcommand{\cW}{{\mathcal W}}

\newcommand{\Sym}{{\mathrm {Sym}}}

\newcommand{\Hom}{{\mathrm {Hom}}}

\title{Extremal $p$-adic $L$-functions}
\author{Santiago Molina Blanco}

\begin{document}
%%%%%%%%%%%%%%%%%%%%%%%%%%%%%%%%%%%%%%%%%%%%%

\maketitle
%%%%%%%%%%%%%%%%%%%%%%%%%%%%%%%%%%%%%%%%%%%%%
%\vskip 1cm

\begin{abstract}
In this note we propose a new construction of cyclotomic $p$-adic L-functions attached to classical modular cuspidal eigenforms. This allows us to cover most known cases to date and provides a method which is amenable to generalizations to automorphic forms on arbitrary groups. In the classical setting of $\GL_2$ over $\Q$ this allows us to construct the $p$-adic $L$-function in the so far uncovered {\em extremal} case which arises under the unlikely hypothesis that  $p$-th Hecke polynomial has a double root. Although Tate's conjecture implies that this case should never take place for $\GL_2/\Q$, the obvious generalization does exist in nature for Hilbert cusp forms over totally real number fields of even degree and this article proposes a method which should adapt to this setting.

We further study the admissibility and the interpolation properties of these \emph{extremal $p$-adic L-functions} $L_p^{\rm ext}(f,s)$, and relate $L_p^{\rm ext}(f,s)$ to the two-variable $p$-adic L-function interpolating cyclotomic $p$-adic L-functions along a Coleman family. 
%Although the fulfillment of this hypothesis contradicts Tate's conjecture in this classical setting, we focus our attention on these \emph{extremal $p$-adic L-functions} because that scenario should not be ruled out as long as the conjecture remains open. Moreover, there are examples of Hilbert modular forms where these extremal hypothesis is satisfied and our work will provide new explicit $p$-adic L-functions in the Hilbert case.
\end{abstract}

\tableofcontents

\section{Introduction}

Let $f \in S_{k+2}(\Gamma_1(N),\epsilon)$ be a modular cuspidal eigeform for $\Gamma_1(N)$ with nebentypus $\epsilon$ and weight $k+2$. A very important topic in modern Number Theory is the study of the complex L-function $L(s,\pi)$ attached to the automorphic representation $\pi$ of $\GL_2(\A)$ generated by $f$. Understanding this complex valued analytic function is the key point for some of the most important problems in mathematics such as the \emph{Birch and Swinnerton-Dyer conjecture}.% or the \emph{Langlands program}.

Back in the middle of the seventies, Vishik \cite{Vis} and Amice-V\'elu \cite{A-V} defined a $p$-adic measure $\mu_{f,p}$ of $\Z_p^\times$ associated with $f$, under the hypothesis that $p$ does not divide $N$. 
The construction of this measure was the starting point for the theory of $p$-adic L-functions attached to modular cuspforms. 
The $p$-adic L-function $L_p(f,s)$ %attached to $f$ 
is a $\C_p$-valued analytic function which interpolates the critical values of the L-function $L(s,\pi)$. The function $L_p(f,s)$ is defined by means of $\mu_{f,p}$ as
\[
L_p(f,s):=\int_{\Z_p^\times}{\rm exp}(s\cdot{\rm log(x)})d\mu_{f,p}(x),
\]
where ${\rm exp}$ and ${\rm log}$ are respectively the $p$-adic exponential and $p$-adic logarithm functions.

Mazur, Tate and Teitelbaum extended in \cite{MTT86} the definition of $\mu_{f,p}$ to more general situations and proposed a $p$-adic analogue of the Birch and Swinnerton-Dyer conjecture, replacing the complex L-function $L(s,\pi)$ with its $p$-adic counterpart $L_p(f,s)$.
It has been shown that $L_p(f,s)$ is directly related with the ($p$-adic, or eventually $l$-adic) cohomology of modular curves, and this makes the $p$-adic Birch and Swinnerton-Dyer conjectures become more tractable. 
In fact, the theory of $p$-adic L-functions has grown tremendously during the last years. Many results, whose complex counterparts are inaccessible with current techniques, have been proven in the analogous $p$-adic scenarios.

In this note we provide a reinterpretation of the construction of the $p$-adic measures $\mu_{f,p}$. Our approach exploits the theory of automorphic representations and, in that sense, it is similar to the construction provided by Spiess in \cite{Spi14}  for weights strictly greater than $2$. This opens the door to possible generalizations of $p$-adic measures attached to automorphic representations of $\GL_2(\A_F)$ of any weight, for any number field $F$.

We are able to construct $\mu_{f,p}$ in every possible situation except when the local automorphic representation $\pi_p$ attached to $f$ is \emph{supercuspidal}, and 
we hope  our work clarifies why it is not expected to find good $p$-adic measures in the latter case. 
%Nevertheless, although all cases where we construct $\mu_{f,p}$ were not considered in \cite{MTT86}, almost all of them can be obtained as twists of the $p$-adic measures described there.

We obtain a genuinely new construction in the unlikely setting where the $p$-th Hecke polynomial has a double root. In this case, our main result (Theorem \ref{mainthm}) reads as follows:
\begin{theorem}
Let $f =\sum_{n\geq 1} a_n q^n\in S_{k+2}(\Gamma_1(N),\epsilon)$ be a cuspform, and assume that $P(X):=X^2-a_pX+\epsilon(p)p^{k+1}$ has a double root $\alpha$. Then there exists a locally analytic $p$-adic measure $\mu_{f,p}^{{\rm ext}}$ of $\Z_p^\times$ such that, for any locally polynomial character $\chi=\chi_0(x)x^m$ with $m\leq k$:
\begin{equation}\label{IntForInt}
\int_{\Z_p^\times}\chi d\mu_{f,p}^{{\rm ext}}=\frac{4\pi }{\Omega_f^\pm i^m}\cdot e_p^{\rm ext}(\pi_p,\chi_0)\cdot L\left(m-k+\frac{1}{2},\pi,\chi_0\right).
\end{equation}

Here $L\left(s,\pi,\chi_0\right)$ denotes the complex the $L$-function of $\pi$ twisted by $\chi_0$, and we have set
\[
e_p^{\rm ext}(\pi_p,\chi_0)=\left\{\begin{array}{ll}
(1-p^{-1})^{-1}\left(p^{k-m}\alpha^{-1}+p^{m-k-1}\alpha-2p^{-1}\right);&\chi_0\mid_{\Z_p^\times}=1;\\
-(1-p^{-1})^{-1}rp^{r(m-k-1)}\alpha^{r}\tau(\chi_0);&{\rm cond}(\chi_0)=r>0,
\end{array}\right.
\]
where $\tau(\chi_0)$ is the Gauss sum attached to $\chi_0$
%of Definition \ref{defGS}.
\end{theorem}

We call $\mu_{f,p}^{{\rm ext}}$ \emph{the extremal $p$-adic measure}.
Coleman and Edixhoven showed in \cite{ColEd} that $P(X)$ never has double roots if the weight is 2, namely, $k = 0$. Moreover, they showed that assuming \emph{Tate's conjecture} the polynomial $P(X)$ can never be a square for general weights $k+2$. Since we believe in Tate's conjecture, we expect this situation never occur, hence surely the hypothesis of the theorem is never fulfilled and %it is expected that 
$\mu_{f,p}^{\rm ext}$ can never be constructed. Since these extremal scenarios do appear in nature for other reductive groups, for instance for $\GL_2/F$ where $F$ is a totally real number field of even degree over $\Q$  (see \cite[\S 3.3.1]{Chi15}), we believe our result above is potentially powerful. We plan to employ the approach of this note to cover these cases in the near future.

Notice that in the unlikely situation of the above theorem, the two $p$-adic measures $\mu_{f,p}$ and $\mu_{f,p}^{{\rm ext}}$ coexist. One can thus define the $p$-adic L-function
\[
L_p^{\rm ext}(f,s):=\int_{\Z_p^\times}{\rm exp}(s\cdot{\rm log(x)})d\mu_{f,p}^{{\rm ext}}(x),
\]
called \emph{the extremal $p$-adic L-function}, which coexists with $L_p(f,s)$, and satisfies the interpolation property \eqref{IntForInt} with completely different Euler factors $e_p^{\rm ext}(\pi_p,\chi_0)$ from the classical scenario. 
%One can think that maybe the existence of $L_p^{\rm ext}(f,s)$ could provide a contradiction that would imply that the Hecke polynomial never has multiple roots, thus providing unconditional proof of this fact studied in \cite{ColEd}.

%Although the extremal situation of $P(x)$ being a square is discarded by Tate's conjecture in this classical setting, there are examples of Hilbert modular forms satisfying this hypothesis (see \cite[\S 3.3.1]{Chi15}). Since our techniques can be %extended to the Hilbert setting following the work of Spiess in \cite{Spi14}, new constructions of concrete extremal $p$-adic L-functions could be provided in these cases. 

In the non-critical setting, namely when the roots of the Hecke polynomial are distinct, there is a classical result that relates $\mu_{f,p}$ to a two-variable $p$-adic L-function $\cL_p$ that interpolates $\mu_{g,p}$ as $g$ ranges over a Coleman family passing through $f$. In \cite{BW}, Betina and Williams have recently extended this result to this critical setting. They construct an element
\[
\cL_p\in T\hat\otimes_{\Q_p}\cR,
\] 
where $\cR$ is the $\Q_p$-algebra of locally analytic distributions of $\Z_p^\times$ and $T$ is certain Hecke algebra defining a connected component of the eigencurve. Since an element of the Coleman family corresponds to a morphism $g:T\rightarrow\bar\Q_p$, the function $\cL_p$ is characterized by the property 
\[
\cL_p=C(g)\cdot\mu_{g,p},
\]
where $C(g)\in \bar\Q_p^\times$ is a constant normalized so that $C(f)=1$.
The following result proved in \S \ref{secLmu} relates $\cL_p$ to our extremal $p$-adic measure $\mu_{f,p}^{{\rm ext}}$:
\begin{theorem}
Let $t\in T$ the element corresponding to $U_p-\alpha$. Then
\[
\frac{\partial \cL_p}{\partial t}(f)\in  \alpha^{-1}\mu_{f,p}^{{\rm ext}}+\bar\Q_p\mu_{f,p}.
\]
\end{theorem}

This last result implies that these extremal $p$-adic L-functions are analogous to the so-called secondary $p$-adic L-functions defined by Bella\"iche in \cite{Be}.

\subsubsection*{Acknowledgements.}
The author would like to thank David Loeffler, V\'ictor Rotger and Chris Williams for their comments and discussions throughout the development of this paper.

The author is supported in part by DGICYT Grant MTM2015-63829-P.
This project has received funding from the European Research Council
(ERC) under the European Union's Horizon 2020 research and innovation
programme (grant agreement No. 682152).

\subsection{Notation}

For any ring $R$,
we denote by 
$\cP(k)_R:=\Sym^{k}(R^2)$ the $R$-module of homogeneous polynomials in two variables with coefficients in $R$, endowed with an action of $\GL_2(R)$:
\begin{equation}\label{actpol}
\left( \left(\begin{array}{cc}a&b\\c&d\end{array}\right)\ast P\right)(x,y):=%(ad-bc)^{-\frac{k}{2}}\cdot 
P\left((x,y)\left(\begin{array}{cc}a&b\\c&d\end{array}\right)\right).
\end{equation}
We denote by $V(k)_R:={\rm Hom}_R(\cP(k)_R,R)$ and $V(k):=V(k)_\C$.
Similarly, we define the (right-) action of $A\in\GL_2(\R)^+$ on the set of modular forms of weight $k+2$
\[
(f\mid A)(z):=\rho(A,z)^{k+2}\cdot f(Az);%\quad\left(\begin{array}{cc}a&b\\c&d\end{array}\right)z:=\frac{az+b}{cz+d},
\qquad \rho\left(\left(\begin{array}{cc}a&b\\c&d\end{array}\right),z\right):=\frac{(ad-bc)}{cz+d}.
\]
%If $f\in M_{k+2}(N)$ and $p$ does not define $N$, we define the Hecke operator $T_p$ as
%\[
%T_pf:=\frac{1}{p}\left(f\mid\left(\begin{array}{cc}p&\\&1\end{array}\right)+\sum_{c\in \Z/p\Z}f\mid\left(\begin{array}{cc}1&c\\&p\end{array}\right)\right).
%\]
%Similarly, if $p\parallel N$ we define
%\[
%U_pf:=\frac{1}{p}\left(\sum_{c\in \Z/p\Z}f\mid\left(\begin{array}{cc}1&c\\&p\end{array}\right)\right).
%\]

We will denote by $dx$ the Haar measure of $\Q_p$ so that ${\rm vol}(\Z_p)=1$. Similarly, we write $d^\times x$ for the Haar measure of $\Q_p^\times$ so that ${\rm vol}(\Z_p^\times)=1$. By abuse of notation, will will also denote by $d^\times x$ the corresponding Haar measure of the group of ideles $\A^\times$.   

For any local character $\chi:\Q_p^\times\rightarrow\C^\times$, write
\[
L(s,\chi)=\left\{\begin{array}{lc}(1-\chi(p)p^{-s})^{-1},&\chi\mbox{ unramified}\\
1,&\mbox{otherwise.}\end{array}\right.
\]

\section{Local integrals}

\subsection{Gauss sums}

In this section $\psi:\Q_p\rightarrow\C^\times$ will be a non-trivial additive character such that $\ker(\psi)=\Z_p$.

\begin{lemma}\label{intpsi}
For all $s\in\Q_p^\times$ and $n>0$, we have
\[
\int_{s+p^n\Z_p}\psi(ax)dx=p^{-n}\psi(sa)\cdot1_{\Z_p}(p^na).
\]
In particular,
\[
\int_{\Z_p^\times}\psi(ax)dx=\left\{\begin{array}{ll}(1-p^{-1}),&a\in\Z_p\\-p^{-1},&a\in p^{-1}\Z_p^\times\\0,&\mbox{otherwise}\end{array}\right.
\]
\end{lemma}
\begin{proof}
We compute
\begin{eqnarray*}
\int_{s+p^n\Z_p}\psi(xa)d x&=&\int_{p^n\Z_p}\psi((s+x)a)d x=\psi(sa)\int_{\Z_p}|xp^n|\psi(xp^na)d^\times x\\
&=&p^{-n}\psi(sa)\int_{\Z_p}\psi(xp^na)d x=p^{-n}\psi(sa)\cdot1_{\Z_p}(p^na).
\end{eqnarray*}
To deduce the second part, notice that
\[
\int_{\Z_p^\times}\psi(ax)dx=\sum_{s\in(\Z/p\Z)^\times}\int_{s+p\Z_p}\psi(ax)dx=p^{-1}\sum_{s\in(\Z/p\Z)^\times}\psi(sa)1_{\Z_p}(pa),
\]
hence the result follows.
\end{proof}
\begin{lemma}\label{psichi}
For all $\chi:\Z_p^\times\rightarrow\C^\times$ be a character of conductor $n\geq 1$. Let $1+p^n\Z_p\subset U\subseteq \Z_p^\times$ be a open subgroup. We have
\[
\int_{U}\chi(x)\psi(ax)d^\times x=0,\qquad\mbox{unless $|a|=p^n$.}
\]
\end{lemma}
\begin{proof}
We compute
\begin{eqnarray*}
\int_{U}\chi(x)\psi(ax)d^\times x&=&\sum_{s\in U/(1+p^n\Z_p)}\chi(s)\int_{s+p^n\Z_p}\psi(ax)d x\\
&=&p^{-n}1_{\Z_p}(p^na)\sum_{s\in U/(1+p^n\Z_p)}\chi(s)\psi(sa).
\end{eqnarray*}
Hence the integral $I:=\int_{U}\chi(x)\psi(ax)d^\times x$ must be zero if $a\not\in p^{-n}\Z_p$. Moreover, if $a\in p^{-n+1}\Z_p$,
\[
I=\int_{U}\chi(x(1+p^{n-1}))\psi(ax(1+p^{n-1}))d^\times x=\chi(1+p^{n-1})I=0,
\]
and the result follows.
\end{proof}
We now define the Gauss sum:
\begin{definition}\label{defGS}
For any character $\chi$ of conductor $n\geq 0$, 
\[
\tau(\chi)=\tau(\chi,\psi)=p^n\int_{\Z_p^\times}\chi(x)\psi(-p^{-n}x)dx.
\]
\end{definition}

\section{Classical cyclotomic $p$-adic $L$-function}\label{Classical}

\subsection{Classical Modular symbols}

Let $f\in S_{k+2}(N,\epsilon)$ be a modular cuspidal newform of weight $(k+2)$ level $\Gamma_1(N)$ and nebentypus $\epsilon$. 

By definition, we have
\[
(f\mid A)(z)\cdot (A^{-1}P)(1,-z)\cdot dz=\det(A)\cdot f(Az)\cdot P(1,-Az)\cdot d(Az), \quad A\in \GL_2(\R)^+,
\]
for any $P\in V(k)$. Hence, if we denote by $\Delta_0$ the group of degree zero divisors of $\PP^1(\Q)$ with the natural action of $\GL_2(\Q)$, we obtain the \emph{Modular Symbol}:
\begin{eqnarray*}
&&\phi_f^{\pm}\in \Hom_{\Gamma_1(N)}(\Delta_0,V(k));\\ 
&&\phi_f^{\pm}(s-t)(P):=2\pi i\left(\int_t^s f(z)P(1,-z)dz\pm\int_{-t}^{-s} f(z)P(1,z)dz\right).
\end{eqnarray*}
Notice that $\Gamma_1(N)$-equivariance follows from relation
\begin{equation}\label{eqint}
\phi_{f\mid A}^\pm(D)=\det(A)\cdot A^{-1}\left(\phi_f^\pm(AD)\right),\qquad A\in\GL_2(\R)^+,
\end{equation}
deduced from the above equality and the fact that {\tiny $\left(\begin{array}{cc}1&\\&-1\end{array}\right)$} normalizes $\Gamma_1(N)$.
The following result is well known and classical:
\begin{proposition}\label{ratMS}
There exists periods $\Omega_\pm$ such that
\[
\phi_f^\pm=\Omega_{\pm}\cdot\varphi_f^\pm,
\]
for some $\varphi_f^\pm\in  \Hom_{\Gamma_1(N)}(\Delta_0,V(k)_{R_f})$, where $R_f$ is the ring of coefficients of $f$.
\end{proposition}

\subsection{Classical $p$-adic distributions}\label{classicdist}

Given $f\in S_{k+2}(N,\epsilon)$, we will assume that $f$ is an eigenvector for the Hecke operator $T_p$ with eigenvalue $a_p$. % and $\varepsilon_p:=\varepsilon\mid_{(\Z/p\Z)^\times}$ is trivial. 
%In case that $p\nmid N$, 
Let $\alpha$ be a non zero root of the Hecke polynomial $X^2-a_pX+\epsilon(p)p^{k+1}$%. In case $p\mid N$, let $\alpha$ be the eigenvalue of the Hecke operator $U_p$. %(this value can be $\pm 1$). 

We will construct distributions $\mu_{f,\alpha}^\pm$ of locally polynomial functions of $\Z_p^\times$ of degree less that $k$ attached to $f$ (and $\alpha$ in case $p\nmid N$). Since the open sets $U(a,n)=a+p^n\Z_p$ ($a\in \Z_p^\times$ and $n\in \N$) form a basis of $\Z_p^\times$, it is enough to define the image of $P\left(1,\frac{x-a}{p^n}\right)1_{U(a,n)}(x)$, for any $P\in \cP(k)_\Z$
\begin{equation}\label{eqclassmu}
\int_{U(a,n)}P\left(1,\frac{x-a}{p^n}\right)d\mu^\pm_{f,\alpha}(x):=\frac{1}{\alpha^n}\varphi^\pm_{f_\alpha}\left(\frac{a}{p^n}-\infty\right)(P),
\end{equation}
where $f_{\alpha}(z):=f(z)-\beta\cdot f(pz)$ and $\beta=\frac{\epsilon(p)p^{k+1}}{\alpha}$.
It defines a distribution because $\mu^\pm_{f,\alpha}$ satisfies \emph{additivity}, namely, since 
\[
P\left(1,\frac{x-a}{p^n}\right)1_{U(a,n)}(x)=%p^{\frac{k}{2}}
\sum_{b\equiv a\;{\rm mod} \;p^{n}} (\gamma_{a,b}P)\left(1,\frac{x-b}{p^{n+1}}\right)1_{U(b,n+1)}(x),\quad\gamma_{a,b}:=\mbox{\tiny$\left(\begin{array}{cc}1&\frac{b-a}{p^n}\\0&p\end{array}\right)$},
\]
and by \eqref{eqint} we have that $U_p\varphi_{f_\alpha}^\pm=\alpha\cdot \varphi_{f_\alpha}^\pm$, where
\begin{equation}\label{defUp}
(U_p\varphi_{f_\alpha}^\pm)(D):=\sum_{c\in \Z/p\Z}\left(\begin{array}{cc}1&c\\&p\end{array}\right)^{-1}\varphi_{f_\alpha}^\pm\left(\left(\begin{array}{cc}1&c\\&p\end{array}\right)D\right),
\end{equation}
it can be shown that 
\[
\int_{U(a,n)}P\left(1,\frac{x-a}{p^n}\right)d\mu^\pm_{f,\alpha}(x)=%p^{\frac{k}{2}}
\sum_{b\equiv a\;{\rm mod} \;p^{n}}\int_{U(b,n+1)} (\gamma_{a,b}P)\left(1,\frac{x-b}{p^{n+1}}\right)d\mu^\pm_{f,\alpha}(x).
\]
%Indeed, we have that $f_{\alpha_p}\in S_{k+2}({\rm lcm}(N,p))$ is a form that satisfies $U_p f_{\alpha_p}=\alpha_p\cdot f_{\alpha_p}$, thus, 
%\begin{eqnarray*}
%\int_{U(a,n)}P\left(1,\frac{x-a}{p^n}\right)d\mu^\pm_{f,\alpha_p}(x)&=&\frac{1}{p\alpha_p}\sum_{\gamma_{a,b}}\frac{1}{\alpha_p^n}\varphi^\pm_{f_{\alpha_p}\mid \gamma_{a,b}}\left(\frac{a}{p^n}-\infty\right)(P)\\
%&=&\sum_{b\equiv a\;{\rm mod} \;p^{n}}\frac{1}{\alpha_p^{n+1}}\varphi^\pm_{f_{\alpha_p}}\left(\gamma_{a,b}\frac{a}{p^n}-\gamma_{a,b}\infty\right)(\gamma_{a,b}P)\\
%&=&\sum_{b\equiv a\;{\rm mod} \;p^{n}}\frac{1}{\alpha_p^{n+1}}\varphi^\pm_{f_{\alpha_p}}\left(\frac{b}{p^{n+1}}-\infty\right)(\gamma_{a,b}P)
%\end{eqnarray*}
%where the second equality follows from \eqref{eqint}.

The following result shows that, under certain hypothesis, we can extend $\mu^\pm_{f,\alpha}$ to a locally analytic measure.
\begin{theorem}[Visnik, Amice-V\'elu]\label{ThmVAV}
Fix an integer $h$ such that $1\leq h\leq k+1$. Suppose that %the polynomial $X^2-a_pX+p^{k+1}$ has a root 
 $\alpha$ satisfies ${\rm ord}_p\alpha<h$. Then there exists a locally analytic measure $\mu_{f,\alpha}^\pm$ satifying:
\begin{itemize}
\item $\int_{U(a,n)}P\left(1,\frac{x-a}{p^n}\right)d\mu^\pm_{f,\alpha}(x):=\frac{1}{\alpha^n}\varphi^\pm_{f_{\alpha}}\left(\frac{a}{p^n}-\infty\right)(P)$, for any locally polynomial function $P\left(1,\frac{x-a}{p^n}\right)1_{U(a,n)}(x)$ of degree strictly less than $h$.

\item For any $m\geq 0$,
\[
\int_{U(a,n)}(x-a)^md\mu^\pm_{f,\alpha}(x)\in \left(\frac{p^m}{\alpha}\right)^n\alpha^{-1}.
\]

\item If $F(x)=\sum_{m\geq 0}c_m(x-a)^m$ is convergent on $U(a,n)$, then
\[
\int_{U(a,n)}F(x)d\mu^\pm_{f,\alpha}(x)=\sum_{m\geq 0}c_m\int_{U(a,n)}(x-a)^md\mu^\pm_{f,\alpha}(x).
\]
\end{itemize}
\end{theorem}
If we assume that there exists such a root $\alpha$ with ${\rm ord}_p\alpha<k+1$, then we define $\mu_{f,\alpha}:=\mu_{f,\alpha}^++\mu_{f,\alpha}^-$ and the (\emph{cyclotomic}) \emph{$p$-adic $L$-function}:
\[
L_p(f,\alpha,s):=\int_{\Z_p^\times}{\rm exp}(s\cdot{\rm log(x)})d\mu_{f,\alpha}(x).
\]
\begin{remark}
Write $V_f$ the $\bar\Q[\GL_2(\Q)]$-representation generated by $f$. For any $g\in V_f$, write 
\begin{equation}\label{defMS}
\varphi_{g}^\pm(s-t)(P):=\frac{2\pi i}{\Omega_\pm}\left(\int_t^s g(z)P(1,-z)dz\pm\int_{-t}^{-s} g(z)P(1,z)dz\right).
\end{equation}
Relation \eqref{eqint} implies that the morphism
\begin{equation}\label{GL_2-equiv}
\varphi^\pm:V_f\longrightarrow \Hom\left(\Delta_0,V(k)_{\bar\Q}\right)[\det],\qquad g\mapsto\varphi_g^\pm,
\end{equation}
is $\GL_2(\Q)$-equivariant.
\end{remark}

\section{$p$-adic $L$-functions}\label{padicLfunct}

In this section we provide a reinterpretation of the distributions $\mu^\pm_{f,\alpha_p}$. 
Let $f\in S_{k+2}(\Gamma_1(N),\epsilon)$ be a cuspidal newform as above and let $p$ be any prime. Fix the embedding
\begin{equation}\label{actQ_p}
\Z_p^\times\hookrightarrow\Q_p^\times\hookrightarrow\GL_2(\Q_p);\qquad x\longmapsto\left(\begin{array}{cc}x&\\&1\end{array}\right).
\end{equation}

\begin{assumption}\label{mainassumption}
Assume that there exists a $\Z_p^\times$-equivariant morphisms 
\[
\delta:C(\Z_p^\times,L)\longrightarrow V,
\]
where $L$ is certain finite extension of the coefficient field $\Q(\{a_n\}_n)$,  and $V$ is certain model over $L$ of the local automorphic representation $\pi_p$ generated by $f$. Assume also that, for big enough $n$,
\begin{equation}\label{keyprop}
\left(\begin{array}{cc}1&s\\&p^n\end{array}\right)\delta(1_{U(s,n)})=\frac{1}{\gamma^n}\sum_{i=0}^mc_i(s,n)V_i,
\end{equation}
where $m$ is fixed, $V_i\in V$ do not depend neither $s$ nor $n$, and $c_i(s,n)\in \cO_L$.
\end{assumption}

\subsection{$p$-adic distributions}

Let us consider the subgroup
\[
\hat K_1(N)=\left\{g\in \GL_2(\hat \Z):\;g\equiv(\begin{smallmatrix}\ast&\ast\\0&1\end{smallmatrix})\;{\rm mod}\;N\right\}.
\]
Again by strong approximation we have that $\GL_2(\A_f)=\GL_2(\Q)^+\hat K_1(N)$.
Thus, for any $\GL_2(\A_f)\ni g=h_gk_g$, where $h_g\in \GL_2(\Q)^+$, $k_g\in \hat K_1(N)$ are well defined up to multiplication by $\Gamma_1(N)=\GL_2(\Q)^+\cap\hat K_1(N)$.
Write $K:=\hat K_1(N)\cap\GL_2(\Z_p)$. By strong multiplicity one $\pi_p^K$ is one dimensional. Therefore $V^K=Lw_0$ and $V=L[\GL_2(\Q_p)]w_0$.
Notice that we have a natural morphism 
\[
\varphi_{f,p}^\pm:V\longrightarrow\Hom(\Delta_0,V(k)_L);\qquad \varphi_{f,p}^\pm(gw_0)=\det(h_g)\cdot
\varphi_{f\mid h_g^{-1}}^\pm.
\] 
\begin{remark}\label{rmkacthom}
If $g\in\GL_2(\Q_p)$ then $h_g\in \hat K_1(N)^p:=\hat K_1(N)\cap\prod_{\ell\neq p}\GL_2(\Q_\ell)$. This implies that, for any $h\in \GL_2(\Q)^+\cap\hat K_1(N)^p$, we have $h_{hg}=h\cdot h_g$ for all $g\in\GL_2(\Q_p)$.
By \eqref{eqint}, this implies that $\varphi_{f,p}^\pm(hv)=h\ast\varphi_{f,p}^\pm(v)$, for all $v\in V\subset\pi_p$, where the action of $h\in \GL_2(\Q)^+\cap\hat K_1(N)^p$ is given by 
\[
(h\ast\varphi)(D):=h(\varphi(h^{-1}D)),\qquad\varphi\in \Hom(\Delta_0,V(k)_L).
\]
\end{remark}
\begin{remark}\label{rmkchar}
By definition, for any $\left(\begin{smallmatrix}a&b\\c&d\end{smallmatrix}\right)\in\Gamma_0(N)$, we have 
\[
f\left(\frac{az+b}{cz+d}\right)=\epsilon(d)\cdot(cz+d)^{k+2}f(z),\qquad f\mid\left(\begin{smallmatrix}a&b\\c&d\end{smallmatrix}\right)=\epsilon(d)\cdot  f.
\]
For any $z\in \Q_p^\times$ such that $z=p^nu$ where $u\in\Z_p^\times$, we can choose $d\in\Z$ such that $d\equiv u^{-1}\;{\rm mod}\;N\Z_p$ and $d\equiv p^{n}\;{\rm mod}\;N\Z_\ell$, for $\ell\neq p$. Let us choose $A=(\begin{smallmatrix}a&b\\c&d\end{smallmatrix})\in\Gamma_0(N)$, and we have  
\[
(z,1)=p^nA^{-1}(uA,p^{-n}A)\in\GL_2(\A_f),\qquad (uA,p^{-n}A)\in\hat K_1(N).
\]
This implies that, if $\varepsilon_p$ is the central character of $\pi_p$, 
\[
\varepsilon_p(z)\varphi_{f,p}^\pm(w_0)=\varphi_{f,p}^\pm(zw_0)=\det(p^nA^{-1})\cdot\varphi_{f\mid p^{-n}A}^\pm=p^{-nk}\epsilon(d)\cdot\varphi_{f}^\pm
\]
Hence $\varepsilon_p=\epsilon_p^{-1}|\cdot|^k$, where $\epsilon_p=\epsilon\mid_{\Z_p^\times}$.
\end{remark}

Again let $C_k(\Z_p^\times,\C_p)$ be the space of locally polynomial functions of $\Z_p^\times$ of degree less that $k$. Recall the $\Z_p^\times$-equivariant isomorphism
\begin{equation}\label{imathpol}
\imath:C(\Z_p^\times,\Z)\otimes_\Z \cP(k)_{\C_p}(-k)\longrightarrow C_k(\Z_p^\times,\C_p) ;\qquad h\otimes P\longmapsto P(1,x)\cdot h(x).
\end{equation}
Fixing $L\hookrightarrow\C_p$, we define the distributions $\mu^\pm_{f,\delta}$ attached to $f$ and $\delta$:
\begin{equation}\label{defmugen}
\int_{\Z_p^\times}\imath(h\otimes P)(x)d\mu^\pm_{f,\delta}(x):=\varphi_{f,p}^\pm(\delta(h))(0-\infty)(P).
\end{equation}
%\begin{remark}
%Notice that all our choices are compatible, in the sense that, for all $\gamma\in \Q^\times\cap\Z_p^\times$,
%\begin{eqnarray*}
%\int_{\Z_p^\times}\gamma\imath(h\otimes P)(x)d\mu^\pm_{f,p}(x)
%\end{eqnarray*}
%\end{remark}

\subsection{Admissible Distributions}

We have just constructed a distribution 
\[
\mu_{f,\delta}^\pm:C_k(\Z_p^\times,\C_p)\longrightarrow \C_p.
\]
This section is devoted to extend this distribution to a locally analytic measure $\mu_{f,\delta}^\pm\in\Hom\left(C_{\rm loc-an}(\Z_p^\times,\C_p),\C_p\right)$.

\begin{definition}
Write $v_p:\C_p\rightarrow\Q\cup\{-\infty\}$ the usual normalized $p$-adic valuation. 
For any $h\in \R^+$,
a distribution $\mu\in \Hom(C_k(\Z_p^\times,\C_p),\C_p)$ is \emph{$h$-admissible} if 
\[
v_p\left(\int_{U(a,n)}g d\mu\right)\geq v_p(A)-n\cdot h,
\]
for some fixed $A\in \C_p$, and any $g\in C_k(\Z_p^\times,\cO_{\C_p})$ which is polynomical in a small enough $U(a,n)\subseteq\Z_p^\times$. We will denote previous relation by
\[
\int_{U(a,n)}g d\mu\in A\cdot p^{-n h}\cO_{\C_p}.
\]
\end{definition}

\begin{proposition}\label{propext}
If $h<k+1$, a $h$-admissible the distribution $\mu$ can be extended to a locally analytic measure such that
\[
\int_{U(a,n)}g d\mu\in A\cdot p^{-n h}\cO_{\C_p}, 
\]
for any $g\in C(\Z_p^\times,\cO_{\C_p})$ which is analytic in $U(a,n)$.
\end{proposition}
\begin{proof}
Notice that any locally analytic function is topologically generated by functions of the form $P_m^{a,N}(x):=\left(\frac{x-a}{p^N}\right)^m1_{U(a,N)}(x)$, where $m\in\N$. By definition,  we have defined the values $\mu(P^{a,N}_m)$ when $m\leq k$. If $m>h$, we define $\mu(P_m^{a,N})=\lim_{n\rightarrow\infty}a_n$, where
\[
a_n=\sum_{b\;{\rm mod}\; p^{n};\;b\equiv a\;{\rm mod}\; p^N}\sum_{j\leq h}\left(\frac{b-a}{p^N}\right)^{m-j}\binom{m}{j}p^{j(n-N)}\mu(P_j^{b,n})%\in Ap^{j(n-N)-nh}\cO_{\C_p}.
\]
and the definition agrees with $\mu$ when $h<m\leq k$ because $p^{j(n-N)}\mu(P_j^{b,n})\stackrel{n}{\rightarrow}0$ when $j>h$, hence
\[
\lim_{n\rightarrow\infty}a_n=\sum_{b\;{\rm mod}\; p^{n};\;b\equiv a\;{\rm mod}\; p^N}\sum_{j=0}^m\left(\frac{b-a}{p^N}\right)^{m-j}\binom{m}{j}p^{j(n-N)}\mu(P_j^{b,n})=\mu(P_m^{a,N})
\]
The limit converge because $\{a_n\}_n$ is Cauchy, indeed by additivity
\[
a_{n_2}-a_{n_1}=\sum_{j\leq h}\sum_{b\equiv a\;(p^{n_2})}\sum_{b'\equiv b\;(p^{n_1})}\sum_{k=h+1}^{m}r(k)\binom{k}{j}\left(\frac{b'-b}{p^N}\right)^{k-j}p^{(n_2-N)j}\mu(P_{j}^{b',n_2}),
\]
where $r(k)=\binom{m}{k}\left(\frac{b'-a}{p^N}\right)^{m-k}$.
Since 
\[
\left(\frac{b'-b}{p^N}\right)^{k-j}p^{(n_2-N)j}\mu(P_{j}^{b',n_2})\in A\cdot p^{-Nk}p^{(n_1-n_2)(k-j)}p^{(k-h)n_2}\cO_{\C_p},
\]
we have that $a_{n+1}-a_{n}\stackrel{n}{\rightarrow} 0$. 

It is clear by the definition that $\mu(P_m^{a,N})\in A\cdot p^{-N h}\cO_{\C_p}$ for all $m, a$ and $N$. Moreover, it extends to a locally analytic measure by continuity which is determined by the image of locally polynomial functions of degree at most $h$.
\end{proof}

Notice that, for all $m\leq k$,
\[
P_m^{a,n}(x)=\left(\frac{x-a}{p^n}\right)^m1_{U(a,n)}(x)=\imath\left(1_{U(a,n)}\otimes\left(\frac{Y-aX}{p^n}\right)^mX^{k-m}\right)
\]
Using property \eqref{keyprop} and Remarks \ref{rmkacthom} and Remark \ref{rmkchar}, we compute that 
\begin{eqnarray*}
\int_{\Z_p^\times}P_m^{a,n} d\mu^\pm_{f,p}&=&\varphi_{f,p}^\pm(\delta(1_{U(a,n)}))(0-\infty)\left(\left(\frac{Y-aX}{p^n}\right)^mX^{k-m}\right)\\
&=&\sum_{i=0}^m\frac{c_i(a,n)}{\gamma^n}\cdot\varphi_{f,p}^\pm\left(p^{-n}\left(\begin{smallmatrix}p^n&-a\\&1\end{smallmatrix}\right)V_i\right)(0-\infty)\left(\left(\frac{Y-aX}{p^n}\right)^mX^{k-m}\right)\\
&=&\sum_{i=0}^m\frac{c_i(a,n)}{\varepsilon_p(p)^{n}\gamma^n}\cdot\varphi^\pm_{f,p}(V_i)\left(\frac{a}{p^n}-\infty\right)\left((p^{-n}Y)^m(p^{-n}X)^{k-m})\right)\\
&=&\sum_{i=0}^m\frac{c_i(a,n)}{\gamma^n}\cdot\varphi^\pm_{f,p}(V_i)\left(\frac{a}{p^n}-\infty\right)\left(Y^mX^{k-m}\right).
\end{eqnarray*}
Notice that $\varphi^\pm_{f,p}(V_i)\in \Hom(\Delta_0,V(k)_L)^{\Gamma_1(Np^r)}_\epsilon:=\Hom_{\Gamma_1(Np^r)}(\Delta_0,V(k)_L)_\epsilon$ for some big enough $r\in\N$, where the subindex $\epsilon$ indicates that the action of $\Gamma_1(Np^r)/\Gamma_0(Np^r)$ is given by the character $\epsilon$. By Manin's trick we have that 
\[
\Hom_{\Gamma_1(Np^r)}(\Delta_0,V(k)_L)_\epsilon\simeq\Hom_{\Gamma_1(Np^r)}(\Delta_0,V(k)_{\cO_L})_\epsilon\otimes_{\cO_L} L.
\]
Since $Y^mX^{k-m}\in \cP(k)_{\cO_L}$, $c(a,n)\in\cO_L$ and the functions $P_m^{a,n}$ generate $C_k(\Z_p^\times,\cO_{\C_p})$, we obtain that 
\begin{equation}\label{calcint}
\int_{U(a,n)}gd\mu^\pm_{f,\delta}\in\frac{A}{\gamma^n}\cO_{\C_p},\qquad\mbox{for all }g\in C_k(\Z_p^\times,\cO_{\C_p}),
\end{equation}
and some fixed $A\in L$.
We deduce the following result.
\begin{theorem}\label{thmadm}
Fix an embedding $L\hookrightarrow\C_p$. We have that $\mu^\pm_{f,\delta}$ is $v_p(\gamma)$-admissible. %, where $v:\C_p^\times\rightarrow\Q$ is the valuation such that $v(p)=1$.
\end{theorem}

\begin{definition}
If we assume that $v_p(\gamma)<k+1$, we define the cyclotomic $p$-adic measure attached to $f$ and $\delta$
\[
\mu_{f,\delta}:=\mu_{f,\delta}^++\mu_{f,\delta}^-.
\]
\end{definition}

\subsection{Interpolation properties}

%\subsection{Modular forms vs Automorphic forms}

Given the modular form $f\in S_{k+2}(\Gamma_1(N))$, let us consider the automorphic form $\phi:\GL_2(\Q)\backslash\GL_2(\A)\rightarrow\C$, characterized by its restriction to $\GL_2(\R)^+\times\GL_2(\A_f)$:
\[
\phi(g_\infty,g_f)=\frac{\det\left(\gamma\right)}{\det(g_\infty)}\cdot f\mid\gamma^{-1}g_\infty \left( i\right), \quad g_f=\gamma k\in \GL_2(\Q)^+\hat K_1(N), \quad g_\infty=\left(\begin{smallmatrix}a&b\\c&d\end{smallmatrix}\right).
\]
Given $g\in\GL_2(\Q_p)$, we compute $\varphi^\pm_{f,p}(gw_0)(0-\infty)(Y^mX^{k-m})=$
\begin{eqnarray*}
&=&\det(h_g)\cdot\varphi^\pm_{f\mid h_g^{-1}}(0-\infty)(Y^mX^{k-m})\\
&=&\frac{-2\pi\det(h_g)}{\Omega_f^\pm}\cdot\left(\int_\infty^0 f\mid h_g^{-1}(ix)(-ix)^mdx\pm\int_\infty^0 f\mid h_g^{-1}(ix)(ix)^mdx\right)\\
&=&\frac{2\pi }{\Omega_f^\pm}\cdot\int_{\R^+}x^{m-k}\cdot\phi\left(\left(\begin{smallmatrix}x&\\&1\end{smallmatrix}\right),g\right)d^{\times}x\cdot((-i)^m\pm i^m).
\end{eqnarray*}
This implies that, if we consider the automorphic representation $\pi$ generated by $\phi$, and the $\GL_2(\Q_p)$-equivariant morphism
\[
\phi_f:\pi_p\longrightarrow\pi:\qquad gw_0\longmapsto g\phi,
\]
we have that 
\[
\varphi^\pm_{f,p}(\delta(h))(0-\infty)(Y^mX^{k-m})=\frac{4\pi(-i)^m }{\Omega_f^\pm}\cdot\int_{\R^+}x^{m-k}\cdot\phi_f\left(\delta(h)\right)\left(\left(\begin{smallmatrix}x&\\&1\end{smallmatrix}\right),1\right)d^{\times}x\cdot\left(\frac{1\pm(-1)^m}{2}\right).
\]
Let $H$ be the maximum subgroup of $\Z_p^\times$ such that $h\mid_{sH}$ is constant, for all $sH\in\Z_p^\times/H$. Notice that $h=\sum_{s\in\Z_p^\times/H}h(s)1_{sH}$. Moreover, for all $v\in\pi_p$, the automorphic form $\phi_f(v)$ is $U^p:=\prod_{\ell\neq p}\Z_\ell^\times$-invariant when embedded in $\GL_2(\A_f)$  by means of \eqref{actQ_p}. Hence, if we consider $\varphi_{f,p}:=\varphi_{f,p}^++\varphi_{f,p}^-$, we have $\varphi_{f,p}(\delta(h))(0-\infty)(Y^mX^{k-m})=$
\begin{eqnarray*}
&=&\sum_{sH\in\Z_p^\times/H}\frac{4\pi h(s)}{i^m\Omega_f^\pm}\cdot\int_{\R^+}\int_{U^p}x^{m-k}\phi_f\left(\delta(1_{sH})\right)\left(\left(\begin{smallmatrix}x&\\&1\end{smallmatrix}\right),1,\left(\begin{smallmatrix}t&\\&1\end{smallmatrix}\right)\right)d^{\times}xd^\times t\\
&=&\sum_{sH\in\Z_p^\times/H}\frac{4\pi h(s)}{i^m\Omega_f^\pm}\cdot\int_{\R^+}\int_{U^p}x^{m-k}\phi_f\left(\delta(1_{H})\right)\left(\left(\begin{smallmatrix}x&\\&1\end{smallmatrix}\right),\left(\begin{smallmatrix}s&\\&1\end{smallmatrix}\right),\left(\begin{smallmatrix}t&\\&1\end{smallmatrix}\right)\right)d^{\times}xd^\times t\\
&=&\frac{4\pi}{\Omega_f^\pm{\rm vol}(H)}\cdot\int_{\A^\times/\Q^\times}\tilde h(y)\cdot\phi_f\left(\delta(1_{H})\right)\left(\begin{smallmatrix}y&\\&1\end{smallmatrix}\right)d^{\times}y,
\end{eqnarray*}
where $\tilde h(y)=(-i)^m\cdot h(y_p|y|y_\infty^{-1})\cdot|y|^{m-k}$, for all $y=(y_v)_v\in\A^\times$, and $\Omega_f^\pm$ is $\Omega_f^+$ or $\Omega_f^-$ depending if $m$ is even or odd.

Let $\chi\in C_k(\Z_p^\times,\C_p)$ be a locally polynomial character. This implies that $\chi(x)=\chi_0(x)x^m$, for some natural $m\leq k$ and some locally constant character $\chi_0$. This implies that $\chi=\imath(\chi_0\otimes Y^{m}X^{k-m})$. We deduce that
\[
\int_{\Z_p^\times}\chi(x)d\mu_{f,\delta}(x):=\frac{4\pi }{\Omega_f^\pm i^m{\rm vol}(H)}\cdot\int_{\A^\times/\Q^\times}\tilde\chi_0(y)|y|^{m-k}\phi_f\left(\delta(1_H)\right)\left(\begin{smallmatrix}y&\\&1\end{smallmatrix}\right)d^{\times}y,
\]
where $\tilde\chi_0(y):=\chi_0(y_p|y|y_\infty^{-1})$.

Let $\psi:\A/\Q\rightarrow\C^\times$ be a global additive character and we define the Whittaker model element
\[
W_\delta^H:\GL_2(\A)\longrightarrow\C;\qquad W_\delta^H(g):=\int_{\A/\Q}\phi_f(\delta(1_H))\left(\left(\begin{array}{cc}1&x\\&1\end{array}\right)g\right)\psi(-x)dx.
\]
This element admits a expression $W_\delta^H(g)=\prod_vW_{\delta,v}^H(g_v)$, if $g=(g_v)\in\GL_2(\A)$. Moreover by \cite[Theorem 3.5.5]{Bump}, it provides the \emph{Fourier expansion} 
\[
\phi_f(\delta(1_H))(g)=\sum_{a\in\Q^\times}W_\delta^H\left(\left(\begin{array}{cc}a&\\&1\end{array}\right)g\right).
\]
We compute
\begin{eqnarray*}
\int_{\A^\times/\Q^\times}\tilde\chi_0(y)|y|^{m-k}\phi_f\left(\delta(1_H)\right)\left(\begin{smallmatrix}y&\\&1\end{smallmatrix}\right)d^{\times}y&=&\int_{\A^\times}\tilde\chi_0(y)|y|^{m-k}W_\delta^H\left(\begin{smallmatrix}y&\\&1\end{smallmatrix}\right)d^{\times}y\\
&=&\prod_v \int_{\Q_v^\times}\tilde\chi_0(y_v)|y_v|^{m-k}W_{\delta,v}^H\left(\begin{smallmatrix}y_v&\\&1\end{smallmatrix}\right)d^{\times}y_v.
\end{eqnarray*}
By definition of $\delta$, when $v\neq p$ the element $W_{\delta,v}^H$ correspond to the new-vector, thus by \cite[Proposition 3.5.3]{Bump} 
\[
\int_{\Q_v^\times}\tilde\chi_0(y_v)|y_v|^{m-k}W_{\delta,v}^H\left(\begin{smallmatrix}y_v&\\&1\end{smallmatrix}\right)d^{\times}y_v=L_v\left(m-k+\frac{1}{2},\pi_v,\tilde\chi_0\right),\qquad v\neq p.
\]
We conclude using the results explained in \cite[\S 3.5]{Bump} 
\[
\int_{\Z_p^\times}\chi(x)d\mu_{f,\delta}(x)=\frac{4\pi }{\Omega_f^\pm i^m}\cdot e_\delta(\pi_p,\chi_0)\cdot L\left(m-k+\frac{1}{2},\pi,\tilde\chi_0\right),
\]
where the \emph{Euler factor}
\[
e_\delta(\pi_p,\chi_0)=\frac{L_p\left(m-k+\frac{1}{2},\pi_p,\tilde\chi_0\right)^{-1}}{{\rm vol}(H)}\int_{\Q_p^\times}\tilde\chi_0(y_p)|y_p|^{m-k}W_{\delta,p}^H\left(\begin{smallmatrix}y_p&\\&1\end{smallmatrix}\right)d^{\times}y_p.
\]

\subsection{The morphisms $\delta$}

In this section we will construct morphisms $\delta$ satisfying Assumption \ref{mainassumption}. The only case that will be left is the case when $\pi_p$ is \emph{supercuspidal}. In this situation we will not be able to construct admissible $p$-adic distributions.

Let $\pi_p$ be the local representation. Let $W:\pi_p\rightarrow\C$ be the Whittaker functional, and let us consider the Kirillov model $\mathcal{K}$ given by the embedding
\[
\lambda:\pi_p\hookrightarrow \mathcal{K};\qquad \lambda(v)(y)=W\left(\left(\begin{array}{cc}y&\\&1\end{array}\right)v\right).
\]
Recall that the Kirillov model lies in the space of locally constant functions $\phi:\Q_p^\times\rightarrow\C$ endowed with the action
\begin{equation}\label{eqKir}
\left(\begin{array}{cc}1&x\\&1\end{array}\right)\phi(y)=\psi(xy)\phi(y),\qquad \left(\begin{array}{cc}a&\\&1\end{array}\right)\phi(y)=\phi(ay).
\end{equation}

We construct the $\Z_p^\times$-equivariant morphism
\begin{equation}\label{defdelta}
\delta:C(\Z_p^\times,\C)\longrightarrow \mathcal{K};\qquad \delta(h)(y)=\int_{\Z_p^\times}\Psi(zy)h(z)\psi(-zy)d^\times z,
\end{equation}
for a well chosen locally constant function $\Psi$. Notice that, if $h=1_{H}$ for $H$ small enough 
\[
\delta(h)(y)=\Psi(y)\int_{H}\psi(-zy)d^\times z={\rm vol}(H)\Psi(y),\qquad \mbox{if $|y|<<0$.}
\]
This implies that, in order to choose $\Psi$, we need to control how $\mathcal{K}$ looks like:
\begin{itemize}
\item By \cite[Theorem 4.7.2]{Bump}, if $\pi_p=\pi(\chi_1,\chi_2)$ principal series then $\mathcal{K}$ consists on functions $\phi$ such that $\phi(y)=0$ for $|y|>>0$, and
\[
\phi(y)=\left\{\begin{array}{ll}C_1|y|^{1/2}\chi_1(y)+C_2|y|^{1/2}\chi_2(y),&\chi_1\neq\chi_2,\\C_1|y|^{1/2}\chi_1(y)+C_2v_p(y)|y|^{1/2}\chi_1(y),&\chi_1=\chi_2,\end{array}\right.\qquad |y|<<0,
\]  
for some constants $C_1$ and $C_2$.
%where $v_p:\Q_p^\times\rightarrow\Z$ is the valuation.

\item By \cite[Theorem 4.7.3]{Bump}, if $\pi_p=\sigma(\chi_1,\chi_2)$ a special representation such that $\chi_1\chi_2^{-1}=|\cdot|^{-1}$ then $\mathcal{K}$ consists on functions $\phi$ such that $\phi(y)=0$ for $|y|>>0$, and 
\[
\phi(y)=C|y|^{1/2}\chi_2(y),\qquad |y|<<0,
\]  
for some constant $C$.

\item By \cite[Theorem 4.7.1]{Bump} If $\pi_p$ is supercuspidal then $\mathcal{K}=C_c(\Q_p^\times,\C)$.

\end{itemize}

By Lemma \ref{intpsi} and Lemma \ref{psichi} we have that $\delta(h)(y)=0$ for $y$ with big absolute value. This implies that 
\begin{itemize}
\item In case $\pi_p=\pi(\chi_1,\chi_2)$ with $\chi_1\neq\chi_2$, we can choose 
\[
\Psi=|\cdot|^{1/2}\chi_1\qquad\mbox{or}\qquad\Psi=|\cdot|^{1/2}\chi_2.
\]

\item In case $\pi_p=\pi(\chi_1,\chi_2)$ with $\chi_1=\chi_2$, we can choose 
\[
\Psi=|\cdot|^{1/2}\chi_1\qquad\mbox{or}\qquad\Psi=v\cdot|\cdot|^{1/2}\chi_1.
\]

\item In case $\pi_p=\sigma(\chi_1,\chi_2)$ we have 
\[
\Psi=|\cdot|^{1/2}\chi_2.
\]

\item In case $\pi_p$ supercuspidal it is not possible to choose any $\Psi$.
\end{itemize}

We have to prove whether $\delta$ satisfies the property \eqref{keyprop}:
If $\Psi$ is invariant under the action of $1+p^n\Z_p$,
\begin{eqnarray*}
\left(\begin{smallmatrix}1&a\\&p^n\end{smallmatrix}\right)\delta(1_{U(a,n)})(y)=&=&\left(\begin{smallmatrix}p^{n}&\\&p^n\end{smallmatrix}\right)\left(\begin{smallmatrix}p^{-n}&\\&1\end{smallmatrix}\right)\left(\begin{smallmatrix}1&a\\&1\end{smallmatrix}\right)\delta(1_{U(a,n)})(y)\\
&=&\varepsilon_p(p^{n})\cdot \psi(ap^{-n}y)\cdot\delta(1_{U(a,n)})(p^{-n}y)\\
&=&\varepsilon_p(p)^{n}\cdot \int_{U(a,n)}\Psi(p^{-n}yz)\psi(p^{-n}y(a-z))d^\times z\\
&=&\frac{\varepsilon_p(p)^{n}\cdot\Psi(p^{-n}ya)\cdot |p|^n}{1-p^{-1}}\cdot\int_{\Z_p}\psi(yz)d z\\
&=&\frac{\varepsilon_p(p)^{n}\cdot |p|^n}{1-p^{-1}}\cdot\Psi(p^{-n}ya)\cdot 1_{\Z_p}(y),
\end{eqnarray*}
since $d^\times x=(1-p^{-1})^{-1}|x|^{-1}dx$.
\begin{itemize}
\item If $\Psi$ is a character we deduce the property \eqref{keyprop} with $m=0$, $\gamma=\Psi(p)p\varepsilon_p(p)^{-1}$, $c_0(a,n)=\Psi(a)$ and $V_0=(1-p^{-1})^{-1}\Psi(y) 1_{\Z_p}(y)$. 

\item If $\Psi=v_p\cdot\chi$, with $\chi$ a character, it also satisfies property \eqref{keyprop} with $m=1$, $\gamma=\chi(p)p\varepsilon_p(p)^{-1}$, $c_0(a,n)=-n\chi(a)$, $c_1(a,n)=\chi(a)$, $V_0=(1-p^{-1})^{-1}\chi(y) 1_{\Z_p}(y)$ and $V_1=(1-p^{-1})^{-1}v_p(y)\chi(y) 1_{\Z_p}(y)$.
\end{itemize}

\subsection{Computation Euler factors}

The following result describes the Euler factors in each of the situations:
\begin{proposition}\label{EulerFactors}
We have the following cases:
\begin{itemize}
\item[$(i)$] If $\Psi=|\cdot|^{1/2}\chi_i$ we have that
\[
e_\delta(\pi_p,\chi_0)=\left\{\begin{array}{lc}
\frac{(1-p^{-1})^{-1}p^{r(m-k-\frac{1}{2})}\chi_i(p)^{-r}\tau(\chi_0\chi_i,\psi)}{L(m-k+1/2,\tilde\chi_0\chi_j)L(k-m+1/2,\tilde\chi_0\chi_i^{-1})},&\pi_p=\pi(\chi_i,\chi_j);\\
\frac{(1-p^{-1})^{-1}p^{r(m-k-\frac{1}{2})}\chi_i(p)^{-r}\tau(\chi_0\chi_i,\psi)}{L(k-m+1/2,\tilde\chi_0\chi_i^{-1})},&\pi_p=\sigma(\chi_i,\chi_j),
\end{array}
\right.
\]
where $r$ is the conductor of $\chi_i\chi_0$.

\item[$(ii)$] If $\Psi=v_p\cdot|\cdot|^{1/2}\chi_i$ we have that
\[
e_\delta(\pi_p,\chi_0)=\left\{\begin{array}{ll}
\frac{p^{k-m-\frac{1}{2}}\chi_i(p)+p^{m-k-\frac{1}{2}}\chi_i(p)^{-1}-2p^{-1}}{1-p^{-1}};&\chi_0\chi_i\mid_{\Z_p^\times}=1;\\
\frac{-rp^{r(m-k-\frac{1}{2})}\chi_i(p)^{-r}\tau(\chi_0\chi_i,\psi)}{1-p^{-1}};&{\rm cond}(\chi_0\chi_i)=r>0.
\end{array}\right.
\]
\end{itemize}
\end{proposition}
\begin{proof}
In order to compute the Euler factors $e_\delta(\pi_p,\chi_0)$, we have to compute the local periods
\[
I_\delta:=\frac{1}{{\rm vol}(H)}\int_{\Q_p^\times}\tilde\chi_0(y)|y|^{m-k}W_{\delta,p}^H\left(\begin{smallmatrix}y&\\&1\end{smallmatrix}\right)d^\times y=\frac{1}{{\rm vol}(H)}\int_{\Q_p^\times}\tilde\chi_0(y)|y|^{m-k}\delta(1_H)(y)d^\times y.
\]
Recalling that $\tilde\chi_0$ is $H$-invariant, we obtain
\[
I_\delta=\frac{1}{{\rm vol}(H)}\int_{\Q_p^\times}\tilde\chi_0(y)|y|^{m-k}\int_{H}\Psi(zy)\psi(-zy)d^\times zd^\times y=\int_{\Q_p^\times}\tilde\chi_0(x)|x|^{m-k}\Psi(x)\psi(-x)d^\times x.
\]

In case $(i)$ we have that $\Psi=|\cdot|^{1/2}\chi_i$,
hence by Lemma \ref{intpsi} and Lemma \ref{psichi}
\begin{eqnarray*}
I_\delta&=&\sum_np^{n(k-m-\frac{1}{2})}\chi_i(p)^n\int_{\Z_p^\times}\chi_0(x)\chi_i(x)\psi(-p^nx)d^\times x\\
&=&\left\{\begin{array}{ll}
\sum_{n\geq0}p^{n(k-m-\frac{1}{2})}\chi_i(p)^n-(1-p^{-1})^{-1}p^{m-k-\frac{1}{2}}\chi_i(p)^{-1};&\chi_0\chi_i\mid_{\Z_p^\times}=1;\\
(1-p^{-1})^{-1}p^{r(m-k-\frac{1}{2})}\chi_i(p)^{-r}\tau(\chi_0\chi_i,\psi);&{\rm cond}(\chi_0\chi_i)=r>0
\end{array}\right.\\
&=&\left\{\begin{array}{ll}
(1-p^{-1})^{-1}(1-p^{m-k-\frac{1}{2}}\chi_i(p)^{-1})(1-p^{k-m-\frac{1}{2}}\chi_i(p))^{-1};&\chi_0\chi_i\mid_{\Z_p^\times}=1;\\
(1-p^{-1})^{-1}p^{r(m-k-\frac{1}{2})}\chi_i(p)^{-r}\tau(\chi_0\chi_i,\psi);&{\rm cond}(\chi_0\chi_i)=r>0
\end{array}\right.
\end{eqnarray*}
Since $e_\delta(\pi_p,\chi_0)=L_p(m-k+1/2,\pi_p,\tilde\chi_0)^{-1}\cdot I_\delta$ and 
\begin{eqnarray*}
L_p(s,\pi_p,\tilde\chi_0)&=&\left\{\begin{array}{lc}L(s,\tilde\chi_0\chi_i)\cdot L(s,\tilde\chi_0\chi_j),&\pi_p=\pi(\chi_i,\chi_j),\\
L(s,\tilde\chi_0\chi_i),&\pi_p=\sigma(\chi_i,\chi_j),\end{array}\right.
\end{eqnarray*}
part $(i)$ follows.

In case $(ii)$ we have that $\Psi=v_p\cdot|\cdot|^{1/2}\chi_i$, hence
we compute 
\begin{eqnarray*}
I_\delta&=&\sum_nnp^{n(k-m-\frac{1}{2})}\chi_i(p)^n\int_{\Z_p^\times}\chi_0(x)\chi_i(x)\psi(-p^nx)d^\times x\\
&=&\left\{\begin{array}{ll}
\sum_{n\geq0}np^{n(k-m-\frac{1}{2})}\chi_i(p)^n+(1-p^{-1})^{-1}p^{m-k-\frac{1}{2}}\chi_i(p)^{-1};&\chi_0\chi_i\mid_{\Z_p^\times}=1;\\
-r(1-p^{-1})^{-1}p^{r(m-k-\frac{1}{2})}\chi_i(p)^{-r}\tau(\chi_0\chi_i,\psi);&{\rm cond}(\chi_0\chi_i)=r>0
\end{array}\right.\\
&=&\left\{\begin{array}{ll}
\frac{p^{k-m-\frac{1}{2}}\chi_i(p)+p^{m-k-\frac{1}{2}}\chi_i(p)^{-1}-2p^{-1}}{(1-p^{-1})(1-p^{k-m-\frac{1}{2}}\chi_i(p))^2};&\chi_0\chi_i\mid_{\Z_p^\times}=1;\\
-r(1-p^{-1})^{-1}p^{r(m-k-\frac{1}{2})}\chi_i(p)^{-r}\tau(\chi_0\chi_i,\psi);&{\rm cond}(\chi_0\chi_i)=r>0,
\end{array}\right.
\end{eqnarray*}
where the second equality follows from the identity $\sum_{n>0}nx^n=x(1-x)^{-2}$.
The result then follows.
\end{proof}

\section{Extremal $p$-adic L-functions}

If $\pi_p=\pi(\chi_1,\chi_2)$ or $\sigma(\chi_1,\chi_2)$ with $\chi_1$ unramified, then the Hecke polynomial $X^2-a_pX+\epsilon(p)p^{k+1}=(x-\alpha)(x-\beta)$, where $\alpha=p^{1/2}\chi_1(p)^{-1}$. This implies that if $\gamma=\alpha$ has small enough valuation, we can always construct $v(\alpha)$-admissible distributions $\mu_\alpha^\pm$ and $\mu_\alpha=\mu_\alpha^++\mu_\alpha^-$. In fact, if $\pi_p=\pi(\chi_1,\chi_2)$ and $\chi_2$ is also unramified, we can sometimes construct a second $v_p(\beta)$-admissible distribution $\mu_\beta$. 

By previous computations, the interpolation property implies that, for any locally polynomial character $\chi=\chi_0(x)x^m\in C_k(\Z_p^\times,\C_p)$,
\[
\int_{\Z_p^\times}\chi d\mu_\alpha=\frac{4\pi }{\Omega_f^\pm i^m}\cdot e_p(\pi_p,\chi_0)\cdot L\left(m-k+\frac{1}{2},\pi,\chi_0\right),
\]
with
\[
e_p(\pi_p,\chi_0)=\left\{\begin{array}{ll}
(1-p^{-1})^{-1}(1-\epsilon(p)\alpha^{-1} p^{m})(1-\alpha^{-1}p^{k-m});&\chi_0\chi_2\mid_{\Z_p^\times}=1;\\
(1-p^{-1})^{-1}p^{rm}\alpha^{-r}\tau(\chi_0\chi_2,\psi);&{\rm cond}(\chi_0\chi_2)=r>0.
\end{array}\right.
\]
This interpolation formula coincides (up to constant) with the classical interpolation formula of the distribution $\mu_{f,\alpha}$ defined in \S \ref{classicdist}. Indeed, it is easy to prove that $\varphi^\pm_{f_\alpha}$ is proportional to $\varphi^\pm_{f,p}(V_0)$ (see equation \eqref{UpV0}), hence the fact that $\mu_{f,\alpha}^\pm$ is proportional to $\mu_{\alpha}^\pm$ follows from \eqref{eqclassmu}, \eqref{defmugen} and property \eqref{keyprop}. 
%In case that $\pi_p=\sigma(\chi_1,\chi_2)$ with $\chi_1$ unramified, one can similarly relate $\mu_{\alpha}$ with $\mu_{f,p}^+$. 
In fact, if $\Psi$ is a character, all the the admissible $p$-adic distributions constructed in this paper are twists of the $p$-adic distributions described in \S \ref{classicdist} (also in \cite{MTT86}), hence for those situations we only provide a new interpretation of classical constructions.

The only genuine new construction is for the case $\Psi=v_p\cdot|\cdot|^{1/2}\chi$ and $\pi_p=\pi(\chi,\chi)$.
\begin{theorem}\label{mainthm}
Let $f\in S_{k+2}(\Gamma_1(N),\epsilon)$ be a newform, and assume that $\pi_p=\pi(\chi,\chi)$. Then there exists a $(k+1)/2$-admissible distribution $\mu_{f,p}^{\rm ext}$ of $\Z_p^\times$ such that, for any locally polynomial character $\chi=\chi_0(x)x^m\in C_k(\Z_p^\times,\C_p)$,
\[
\int_{\Z_p^\times}\chi d\mu_{f,p}^{\rm ext}=\frac{4\pi }{\Omega_f^\pm i^m}\cdot e_p^{\rm ext}(\pi_p,\chi_0)\cdot L\left(m-k+\frac{1}{2},\pi,\chi_0\right),
\]
with
\[
e_p^{\rm ext}(\pi_p,\chi_0)=\left\{\begin{array}{ll}
\frac{p^{k-m-\frac{1}{2}}\chi(p)+p^{m-k-\frac{1}{2}}\chi(p)^{-1}-2p^{-1}}{1-p^{-1}};&\chi_0\chi\mid_{\Z_p^\times}=1;\\
\frac{-rp^{r(m-k-\frac{1}{2})}\chi(p)^{-r}\tau(\chi_0\chi,\psi)}{1-p^{-1}};&{\rm cond}(\chi_0\chi)=r>0.
\end{array}\right.
\]
\end{theorem}
\begin{proof}
The only thing that is left to prove is that $\mu_{f,p}^{\rm ext}$ is $(k+1)/2$-admissible, but this follows directly from Theorem \ref{thmadm} and the fact that 
\[
\varepsilon_p=\epsilon_p^{-1}|\cdot|^k=\chi^2,\qquad\gamma=\chi(p)p|p|^{\frac{1}{2}}\varepsilon_p(p)^{-1}=\chi(p)p^{\frac{1}{2}+k}\epsilon_p(p).
\]
Hence $v_p(\gamma)=\frac{1}{2}+k+v_p(\chi(p))=\frac{k+1}{2}$.
\end{proof}

%Coleman and Edixhoven showed in \cite{ColEd} that it is impossible to have $\pi_p=\pi(\chi,\chi)$ if $k = 0$. Moreover, they showed that the property $\pi_p=\pi(\chi,\chi)$ for $k>0$ contradicts Tate's conjecture. Since we believe in Tate's conjecture, we expect $\pi_p=\pi(\chi,\chi)$ never occur and that is the reason why 

\begin{remark}
Notice that $\mu_{f,p}^{\rm ext}$ has been constructed as the sum 
\[
\mu_{f,p}^{\rm ext}=\mu_{f,p}^{{\rm ext},+}+\mu_{f,p}^{{\rm ext},-}.
\]
\end{remark}

\begin{definition}
We call $\mu_{f,p}^{\rm ext}$ \emph{extremal $p$-adic measure}.
Since $(k+1)/2<k+1$, by Proposition \ref{propext} we can extend $\mu_{f,p}^{\rm ext}$ to a locally analytic measure. Hence we define the \emph{extremal $p$-adic L-function}
\[
L_p^{\rm ext}(f,s):=\int_{\Z_p^\times}{\rm exp}(s\cdot{\rm log(x)})d\mu_{f,p}^{\rm ext}(x).
\]
\end{definition}

%In the situation $\epsilon=1$, we have $\alpha=p^{1/2}\chi(p)^{-1}=p^{1/2+k}\chi(p)$. Therefore 
%\[
%e_p^{\rm imp}(\pi_p,\chi_0)=\left\{\begin{array}{ll}
%\frac{p^{k-m}+p^{m-k-1}-2 p^{(k-1)/2}}{p^{(k+1)/2}(1-p^{-1})};&\chi_0\chi\mid_{\Z_p^\times}=1;\\
%\frac{-rp^{r(m-k-1)}\alpha^{r}\tau(\chi_0\chi,\psi)}{1-p^{-1}};&{\rm cond}(\chi_0\chi)=r>0.
%\end{array}\right.
%\]

Hence, we conclude that in the conjecturally impossible situation that $\pi_p=\pi(\chi,\chi)$, two $p$-adic L-functions coexist
\[
L_p(f,s),\qquad L_p^{\rm ext}(f,s).
\] 
their corresponding interpolation properties look similar but they have completely different Euler factors. 

\subsection{Alternative description}

In the classical setting described in \S \ref{Classical} ($\chi$ unramified), $p$-adic distributions $\mu^{\pm}_{f,p}$ are given by Equation \eqref{eqclassmu}, while extremal $p$-adic distributions satisfy
\begin{eqnarray*}
\int_{U(a,n)}P\left(1,\frac{x-a}{p^n}\right)d\mu^{{\rm ext},\pm}_{f,p}(x)&=&\varphi_{f,p}^\pm(\delta(1_{U(a,n)}))(0-\infty)\left(P\left(X,\frac{Y-aX}{p^n}\right)\right)\\
&=&\frac{1}{\alpha^n}\cdot\varphi^\pm_{f,p}(V_1-nV_0)\left(\frac{a}{p^n}-\infty\right)\left(P\right),%\\
%&=&\frac{1}{\alpha_p^n}\varphi^+_{f_\alpha}\left(\frac{a}{p^n}-\infty\right)(P)
\end{eqnarray*}
where $V_0=(1-p^{-1})^{-1}|y|^{1/2}\chi(y)1_{\Z_p}(y)$ and $V_1=(1-p^{-1})^{-1}v_p(y)|y|^{1/2}\chi(y)1_{p\Z_p}(y)$. Using the relations \eqref{eqKir}, we  compute the action of the Hecke operator $T_p$ on $V_0+V_1$:
\begin{eqnarray*}
T_p (V_0+V_1)&=&\left(\begin{array}{cc}p^{-1}&\\&1\end{array}\right)(V_0+V_1)+\sum_{c\in\Z/p\Z}\left(\begin{array}{cc}1&p^{-1}c\\&p^{-1}\end{array}\right)(V_0+V_1)\\
%&=&(V_0+V_1)(p^{-1}y)+\frac{1}{\varepsilon_p(p)}\sum_{c\in\Z/p\Z}\left(\begin{array}{cc}p&\\&1\end{array}\right)\left(\begin{array}{cc}1&p^{-1}c\\&1\end{array}\right)(V_0+V_1)(y)\\
%&=&(V_0+V_1)(p^{-1}y)+\frac{1}{\varepsilon_p(p)}\sum_{c\in\Z/p\Z}\left(\begin{array}{cc}1&p^{-1}c\\&1\end{array}\right)(V_0+V_1)(py)\\
&=&(V_0+V_1)(p^{-1}y)+\frac{1}{\varepsilon_p(p)}(V_0+V_1)(py)\sum_{c\in\Z/p\Z}\psi(cy)\\
&=&\frac{\alpha|y|^{1/2}\chi(y)}{(1-p^{-1})}\left(v_p(y)1_{\Z_p}(p^{-1}y)+\frac{1+v_p(py)}{p}\sum_{c\in\Z/p\Z}\psi(cy)1_{\Z_p}(py)\right)\\
&=&\frac{|y|^{1/2}\chi(y)}{(1-p^{-1})}2\alpha\left(1+v_p(y)\right)1_{\Z_p}(y)=2\alpha(V_0+V_1)
\end{eqnarray*}
since $\alpha=\gamma=p^{1/2}\chi(p)^{-1}=\varepsilon_p(p)^{-1}p^{1/2}\chi(p)$.
Similarly,
\begin{equation}\label{UpV0}
U_p V_0=\sum_{c\in\Z/p\Z}\left(\begin{array}{cc}1&p^{-1}c\\&p^{-1}\end{array}\right)V_0=\frac{1}{\varepsilon_p(p)}V_0(py)\sum_{c\in\Z/p\Z}\psi(cy)=\alpha V_0.%\\
%T_p V_1&=&p\left(\begin{array}{cc}p^{-1}&\\&1\end{array}\right)V_1+p\sum_{c\in\Z/p\Z}\left(\begin{array}{cc}1&p^{-1}c\\&p^{-1}\end{array}\right)V_1\\
%&=&pV_1(p^{-1}y)+\frac{p}{\varepsilon_p(p)}V_1(py)\sum_{c\in\Z/p\Z}\psi(cy)=\alpha (V_0+ V_1)
\end{equation}
Hence, $V_0$ and $V_1$ are basis of the generalized eigenspace of $U_p$, in which $V_0$ is the eigenvector and $V_0+V_1$ is the newform. This implies that (up to constant) $\varphi^\pm_{f,p}(V_0)\stackrel{\cdot}{=}\varphi^\pm_{f_\alpha}$, where $f_\alpha$ is the p-specialization defined in \S \ref{classicdist}, while we have that $\varphi^\pm_{f,p}(V_0+V_1)\stackrel{\cdot}{=}\varphi^\pm_{f}$. We conclude that, in terms of the classical definitions given in \S \ref{classicdist}, the extremal distribution can be described as
\[
\int_{U(a,n)}P\left(1,\frac{x-a}{p^n}\right)d\mu^{{\rm ext},\pm}_{f,p}(x)=\frac{1}{\alpha^n}\cdot\varphi^\pm_{f-(n+1)f_\alpha}\left(\frac{a}{p^n}-\infty\right)\left(P\right).
\]

\section{Overconvergent modular symbols}

For any $r\in p^{\Q}$, let $B[\Z_p,r]=\{z\in\C_p,\;\exists a\in\Z_p,\;|z-a|\leq r\}$. We denote by $A[r]$ the ring of affinoid function on $B[\Z_p,r]$. The ring $A[r]$ has structure of $\Q_p$-Banach algebra with the norm $\parallel f\parallel_r={\rm sup}_{z\in B[\Z_p,r]}|f(z)|$. Denote by $D[r]=\Hom_{\Q_p}(A[r],\Q_p)$ the continuous dual. It is also a Banach space with the norm
\[
\parallel \mu\parallel_r={\rm sup}_{f\in A[r]}\frac{|\mu(f)|}{\parallel f\parallel_r}.
\]
We define 
\[
D^\dagger[r]:=\varprojlim_{r'\in p^\Q,r'>r}D[r'],
\]
where the projective limit is taken with respect the usual maps $D[r_2]\rightarrow D[r_1]$, $r_1>r_2$. Since these maps are injective and compact, the space $D^\dagger[r]$ is endowed with structure of Frechet space.

Given an affinoid $\Q_p$-algebra $R$ and a character $w:\Z_p\rightarrow R^\times$ such that $w\in A[r]\hat\otimes_{\Q_p} R$, we can define an action of the monoid 
\[
\Sigma_0(p)=\left\{\left(\begin{array}{cc}a&b\\c&d\end{array}\right)\in \M_2(\Z_p),\;p\nmid a,\;p\mid c,\;ad-bc\neq 0\right\}
\]
on $A[r]\hat\otimes_{\Q_p} R$ and $D[r]\hat\otimes_{\Q_p} R$ given by
\begin{eqnarray*}
(\gamma\ast_w f)(z)&=&w(a+cz) \cdot f\left(\frac{b+dz}{a+cz}\right),\qquad f\in A[r]\hat\otimes_{\Q_p} R,\\
(\gamma\ast_w \mu)(f)&=&\mu(\gamma^{-1}\ast_wf),\qquad\gamma^{-1}\in\Sigma_0(p),\quad \mu\in D[r]\hat\otimes_{\Q_p} R.
\end{eqnarray*}
Write $D_w[r]$ for the space $D[r]\hat\otimes_{\Q_p} R$ with the corresponding action. Similarly we define 
\[
D_w^\dagger:=\varprojlim_{r'\in p^\Q,r'>r} D_w[r]=D^\dagger[r]\hat\otimes_{\Q_p}R,
\]
by \cite[Lemma 3.2]{Be}. Compatibility with base change and \cite[Lemma 3.5]{Be} imply that, given a morphism of affinoid $\Q_p$-algebras $\varphi: R\rightarrow R'$ we have isomorphisms
\begin{equation}\label{esp}
D_w[r]\otimes_R R'\stackrel{\simeq}{\longrightarrow}D_{\varphi\circ w}[r],\qquad D^\dagger_w[r]\otimes_R R'\stackrel{\simeq}{\longrightarrow}D^\dagger_{\varphi\circ w}[r].
\end{equation}

\begin{definition}
We call the space $\Hom_\Gamma(\Delta_0,D^\dagger_w[r])$ the \emph{space of modular symbols of weight $w$}. We denote by $\Hom^\pm_\Gamma(\Delta_0,D^\dagger_w[r])$ the subgroup of $\Hom_\Gamma(\Delta_0,D^\dagger_w[r])$ of elements
that are fixed or multiplied by $-1$ by the involution given by $\mbox{\tiny$\left(\begin{array}{cc}-1&\\&1\end{array}\right)$}$.
\end{definition}

The action of $\Sigma_0(p)$ on $D^\dagger_w[r]$ induces an action of $U_p$ on $\Hom^\pm_\Gamma(\Delta_0,D^\dagger_w[r])$ given by the formula \eqref{defUp}.

Assume that $R$ is reduced and its norm $|\cdot|$ extends the norm of $\Q_p$. Write as usual $v_p(x)=-\log|x|/\log p$, so that $v_p(p)=1$. Let us consider 
\[
R\{\{T\}\}:=\left\{\sum_{n\geq 0}a_nT^n,\;a_n\in R,\;\lim_n(v_p(a_n)-n\nu)=\infty\mbox{ for all }\nu\in \R\right\}
\]
Given $F(T)\in R\{\{T\}\}$ and $\nu\in \R$,
\[
N(F,\nu):=\max\{n\in\N,\;v_p(a_n)-n\nu={\rm inf}_m(v_p(a_m)-m\nu)\}.
\]
A polynomial $Q(T)\in R[T]\subseteq R\{\{T\}\}$ is $\nu$-dominant if it has degree $N(Q,\nu)$ and, for all $x\in {\rm Sp}(R)$, we have $N(Q,\nu)=N(Q_x,\nu)$. We say that $F(T)\in R\{\{T\}\}$ is \emph{$\nu$-adapted} if there exists a (unique) decomposition
$F(T)=Q(T)\cdot G(T)$, where $Q(T)\in R[T]$ is a $\nu$-dominant polynomial of degree $N(F,\nu)$ and $Q(0)=G(0)=1$.

Since $\Hom^\pm_\Gamma(\Delta_0,D_w[r])$ satisfies property (Pr) of \cite[\S 2]{Buz} and $U_p$ acts compactly, one can define the characteristic power series $F(T)\in R\{\{T\}\}$ of $U_p$ acting on $\Hom^\pm_\Gamma(\Delta_0,D_w[r])$. We say that $R$ is \emph{$\nu$-adapted} for some $\nu\in\R$, if $F$ is $\nu$-adapted. If this is the case, we can define the submodule $\Hom^\pm_\Gamma(\Delta_0,D_w[r])^{\leq\nu}$ of \emph{slope bounded by $\nu$ modular symbols} as the kernel of $Q(U_p)$ in $\Hom^\pm_\Gamma(\Delta_0,D_w[r])$.

We write $\Hom^\pm_\Gamma(\Delta_0,D^\dagger_w[r])^{\leq\nu}$ for the intersection
\[
\Hom^\pm_\Gamma(\Delta_0,D^\dagger_w[r])^{\leq\nu}:=\Hom^\pm_\Gamma(\Delta_0,D^\dagger_w[r])\cap \Hom^\pm_\Gamma(\Delta_0,D_w[r'])^{\leq\nu}
\] 
in $\Hom^\pm_\Gamma(\Delta_0,D_w[r'])$, for any $r'>r$. 

\subsection{Control Theorem}\label{CtrlThm}

Let us consider the character
\[
k:\Z_p^\times\rightarrow\Q_p^\times,\qquad x\longmapsto x^k.
\]
Then we have a morphism of $\Sigma_0(p)$-modules
\[
\rho_k^\ast:D^\dagger_k[1]\longrightarrow V(k):=V(k)_{\Q_p};\qquad\rho_k^\ast(\mu)(P):=\mu(P(1,z)).
\]
This provides a morphism
\begin{equation}\label{morphrho}
\rho_k^\ast:\Hom^\pm_\Gamma(\Delta_0,D^\dagger_k[1])\longrightarrow \Hom^\pm_\Gamma(\Delta_0,V(k))
\end{equation}
\begin{theorem}[Steven's control Theorem]\label{ctrlthm}
The above morphism induces an isomorphism of $\Q_p$-vector spaces 
\[
\rho_k^\ast:\Hom^\pm_\Gamma(\Delta_0,D^\dagger_k[1])^{< k+1}\longrightarrow \Hom^\pm_\Gamma(\Delta_0,V(k))^{< k+1}.
\]
\end{theorem}
\begin{proof}
See \cite[Theorem 7.1]{S} and \cite[Theorem 5.4]{SP}.
\end{proof}

\subsection{Extremal modular symbols}

Let $f\in S_{k+2}(N,\epsilon)$ as before, and assume that the Hecke polynomial $x^2-a_px+\epsilon(p)p^{k+1}$ has a double root $\alpha$.
We have defined admissible locally analytic measures $\mu^{{\rm ext},\pm}_{f,p}$ characterized by
\[
\int_{a+p^n\Z_p}P\left(1,\frac{x-a}{p^n}\right)d\mu^{{\rm ext},\pm}_{f,p}(x)=\frac{1}{\alpha^n}\cdot\varphi^\pm_{f-(n+1)f_\alpha}\left(\frac{a}{p^n}-\infty\right)\left(P\right),
\]
for any $P\in \cP(k)_{\Q}$. Our aim is to describe $\mu^{{\rm ext},\pm}_{f,p}$ as the evaluation at $0-\infty$ of certain overconvergent modular symbol $\Hom^\pm_\Gamma(\Delta_0,D^\dagger_k[0])$.

Notice that, if we write $g_n:=f-(n+1)f_\alpha$ and $\gamma_{a,n}:=\left(\begin{array}{cc}1&a\\&p^n\end{array}\right)$,
\begin{eqnarray*}
\int_{\Z_p}\gamma_{a,n}^{-1}\left(\rho_k(P)1_{\Z_p}\right)(x)d\mu^{{\rm ext},\pm}_{f,p}(x)&=&\int_{a+p^n\Z_p}P\left(1,\frac{x-a}{p^n}\right)d\mu^{{\rm ext},\pm}_{f,p}(x)\\
&=&\frac{1}{\alpha^n}\cdot\varphi^\pm_{g_n}\left(\frac{a}{p^n}-\infty\right)\left(P\right)\\
&=&\frac{1}{\alpha^n}\cdot\varphi^\pm_{g_n}\left(\gamma_{a,n}(0-\infty)\right)\left(P\right)\\
&=&\left(\frac{1}{p\alpha}\right)^n\cdot\varphi^\pm_{g_n\mid_{\gamma_{a,n}}}\left(0-\infty\right)\left(\gamma_{a,n}^{-1}P\right).
\end{eqnarray*}
Moreover, the elements $\gamma_{a,n}^{-1}\left(\rho_k(P)1_{\Z_p}\right)\in A[p^{-n}]$ for all $n\in\N$, $a\in\Z_p$, and these functions form a dense set in $\bigcup_{n\geq 0}A[p^{-n}]$.
\begin{lemma}
For any divisor $D\in \Delta_0$, the expression
\[
\gamma_{a,n}^{-1}\left(\rho_k(P)1_{\Z_p}\right)\longmapsto\left(\frac{1}{p\alpha}\right)^n\cdot\varphi^\pm_{g_n\mid_{\gamma_{a,n}}}\left(D\right)\left(\gamma_{a,n}^{-1}P\right)
\] 
extends to a measure in $\hat\varphi_{\rm ext}^\pm(D)\in D_k^\dagger[1]$.
\end{lemma}
\begin{proof}
we have to show \emph{additivity}, namely, since 
\[
\gamma_{a,n}^{-1}\left(\rho_k(P)1_{\Z_p}\right)=%p^{\frac{k}{2}}
\sum_{b\equiv a\;{\rm mod} \;p^{n}} \gamma_{b,n+1}^{-1}\left(\rho_k(\gamma_bP)1_{\Z_p}\right),\quad\gamma_{b}:=\left(\begin{array}{cc}1&\frac{b-a}{p^n}\\0&p\end{array}\right),
\]
we have to show that 
\[
\left(\frac{1}{p\alpha}\right)^n\cdot\varphi^\pm_{g_n\mid_{\gamma_{a,n}}}\left(D\right)\left(\gamma_{a,n}^{-1}P\right)=%p^{\frac{k}{2}}
\sum_{b\equiv a\;{\rm mod} \;p^{n}}\left(\frac{1}{p\alpha}\right)^{n+1}\cdot\varphi^\pm_{g_{n+1}\mid_{\gamma_{b,n+1}}}\left(D\right)\left(\gamma_{b,n+1}^{-1}\gamma_bP\right).
\]
Indeed, we have that $\gamma_{b,n+1}^{-1}\gamma_b=\gamma_{a,n}^{-1}$, thus the above equation follows from the fact that $g_n\in S_{k+2}(\Gamma,\epsilon)$ satisfies $U_p g_{n+1}=\frac{1}{p}\sum_{b\equiv a}g_{n+1}\mid_{\gamma_b}=\alpha\cdot g_{n}$.

First we notice that by \eqref{eqint}, for any $P\in \cP(k)_{\Z_p}$,
\[
\hat\varphi_{\rm ext}^+(D)(\gamma_{a,N}^{-1}\left(\rho_k(P)1_{\Z_p}\right))=\left(\frac{1}{\alpha}\right)^N\cdot\varphi^+_{g_N}\left(\gamma_{a,N}D\right)\left(P\right)\in A\cdot p^{-N \frac{k}{2}}\cO_{\C_p},
\]
for big enough $N$ since $v_p(\alpha)=k/2$.

On the other hand, any locally analytic function is topologically generated by functions of the form $P_m^{a,N}(x):=\left(\frac{x-a}{p^N}\right)^m1_{a+p^N}(x)$, where $m\in\N$. The functions $\gamma_{a,N}^{-1}\left(\rho_k(P)1_{\Z_p}\right)$ are generated by $P^{a,N}_m$ when $m\leq k$, hence our distribution must be determined by 
\[
\hat\varphi_{\rm ext}^\pm(D)(P^{a,N}_m)=\left(\frac{1}{p\alpha}\right)^N\cdot\varphi^\pm_{g_N\mid_{\gamma_{a,N}}}\left(D\right)\left(\gamma_{a,N}^{-1}(x^{k-m}y^m)\right),\qquad m\leq k. 
\]
If $m>k$, we define $\hat\varphi_{\rm ext}^\pm(D)(P_m^{a,N})=\lim_{n\rightarrow\infty}a_n$, where
\[
a_n=\sum_{b\;{\rm mod}\; p^{n};\;b\equiv a\;{\rm mod}\; p^N}\sum_{j\leq k}\left(\frac{b-a}{p^N}\right)^{m-j}\binom{m}{j}p^{j(n-N)}\hat\varphi_{\rm ext}^\pm(D)(P_j^{b,n}).%\in Ap^{j(n-N)-nh}\cO_{\C_p}.
\]
%and the definition agrees with $\mu$ when $h<m\leq k$ because $p^{j(n-N)}\mu(P_j^{b,n})\stackrel{n}{\rightarrow}0$ when $j>h$, hence
%\[
%\lim_{n\rightarrow\infty}a_n=\sum_{b\;{\rm mod}\; p^{n};\;b\equiv a\;{\rm mod}\; p^N}\sum_{j=0}^m\left(\frac{b-a}{p^N}\right)^{m-j}\binom{m}{j}p^{j(n-N)}\mu(P_j^{b,n})=\mu(P_m^{a,N})
%\]
The limit converge because $\{a_n\}_n$ is Cauchy, indeed by additivity
\[
a_{n_2}-a_{n_1}=\sum_{j\leq h}\sum_{b\equiv a\;(p^{n_2})}\sum_{b'\equiv b\;(p^{n_1})}\sum_{k=h+1}^{m}r(k)\binom{k}{j}\left(\frac{b'-b}{p^N}\right)^{k-j}p^{(n_2-N)j}\hat\varphi_{\rm ext}^\pm(D)(P_{j}^{b',n_2}),
\]
where $r(k)=\binom{m}{k}\left(\frac{b'-a}{p^N}\right)^{m-k}$.
Since 
\[
\left(\frac{b'-b}{p^N}\right)^{k-j}p^{(n_2-N)j}\hat\varphi_{\rm ext}^\pm(D)(P_{j}^{b',n_2})\in A\cdot p^{(n_1-n_2)(k-j)}p^{k\left(\frac{n_2}{2}-N\right)}\cO_{\C_p},
\]
we have that $a_{n+1}-a_{n}\stackrel{n}{\rightarrow} 0$. 
%It is clear by the definition that $\hat\varphi_{\rm ext}^+(D)(P_m^{a,N})\in A\cdot p^{-N h}\cO_{\C_p}$ for all $m, a$ and $N$. 
Hence we have extended $\hat\varphi_{\rm ext}^\pm(D)$ to a locally analytic measure by continuity, which is determined by the image of locally polynomial functions of degree at most $k$.
\end{proof}

The above lemma implies that $\hat\varphi_{\rm ext}^\pm\in \Hom(\Delta_0,D_k^\dagger[1])$. Let us check that it is $\Gamma$-equivariant: 
%Notice that the morphism 
%\[
%h^{-1}\left(\rho_k(P)1_{\Z_p}\right)\longmapsto\left(\frac{1}{p\alpha}\right)^n\cdot\varphi^+_{g_n\mid_{h}}\left(D\right)\left(h^{-1}P\right),\qquad h\in \Sigma_0(p),\;\det(h)=p^n
%\]
%is well defined. Indeed, 
For any $g\in\Gamma$, it is easy to show that $g\gamma_{a,n}^{-1}1_{\Z_p}=\gamma_{g^{-1}a,n}^{-1}1_{\Z_p}$, where $\left(\begin{array}{cc}\alpha&\beta\\\gamma&\delta\end{array}\right)a=\frac{\beta+\delta a}{\alpha+\gamma a}$. Thus by \eqref{GL_2-equiv}
\begin{eqnarray*}
\hat\varphi_{\rm ext}^\pm(g D)(g\gamma_{a,n}^{-1}\left(\rho_k(P)1_{\Z_p}\right))&=&\hat\varphi_{\rm ext}^\pm(g D)(\gamma_{g^{-1}a,n}^{-1}\left(\rho_k(\gamma_{g^{-1}a,n}g\gamma_{a,n}^{-1}P)1_{\Z_p}\right))\\
&=&\left(\frac{1}{p\alpha}\right)^n\cdot\varphi^\pm_{g_n\mid_{\gamma_{g^{-1}a,n}}}\left(gD\right)\left(g\gamma_{a,n}^{-1}P\right)\\
&=&\left(\frac{1}{p\alpha}\right)^n\cdot\varphi^\pm_{g_n\mid_{\gamma_{g^{-1}a,n}g}}\left(D\right)\left(\gamma_{a,n}^{-1}P\right)\\
&=&\hat\varphi_{\rm ext}^\pm(D)(\gamma_{a,n}^{-1}\left(\rho_k(P)1_{\Z_p}\right))
\end{eqnarray*}
where the last equality has been obtained from the fact that $\gamma_{g^{-1}a,n}g\gamma_{a,n}^{-1}\in\Gamma$ and $g_n$ is $\Gamma$-invariant for all $n$.
%Since $\hat\varphi^\pm$ is defined as limits of functions of the form $\gamma_{a,n}^{-1}\left(\rho_k(P)1_{\Z_p}\right)$, 
One easily checks that $\hat\varphi^\pm$ is in the corresponding $\mbox{\tiny$\left(\begin{array}{cc}-1&\\&1\end{array}\right)$}$-subspace
\[
\hat\varphi_{\rm ext}^\pm\in  \Hom^\pm_\Gamma(\Delta_0,D_k^\dagger[1]).
\]
From the definition it is easy to check the following result
\begin{proposition}
The measures $\mu^{{\rm ext},\pm}_{f,p}$ and $\mu^{{\rm ext}}_{f,p}$ can be obtained as
\[
\mu^{{\rm ext},\pm}_{f,p}=\hat\varphi_{\rm ext}^\pm(0-\infty)\mid_{\Z_p^\times},\qquad \mu^{{\rm ext}}_{f,p}=\hat\varphi_{\rm ext}(0-\infty)\mid_{\Z_p^\times},
\]
where $\hat\varphi_{\rm ext}:=\hat\varphi_{\rm ext}^++\hat\varphi_{\rm ext}^-$.
\end{proposition}

\subsection{Action of $U_p$}\label{actionUp}

Recall that the action of $\Sigma_0(p)$ on $\Hom_\Gamma(\Delta_0,D_k^\dagger[1])$ provides an action of the Hecke operator $U_p$, the aim of this section is to compute $U_p\hat\varphi_{\rm ext}^\pm$. Notice that it is enough to compute the image of the functions $f_{a,n,P}:=\gamma_{a,n}^{-1}\left(\rho_k(P)1_{\Z_p}\right)$:
\begin{eqnarray*}
(U_p\hat\varphi_{\rm ext}^\pm)(D)(f_{a,n,P})&=&\sum_{c\;{\rm mod}\;p}\hat\varphi_{\rm ext}^\pm(\gamma_{c,1}D)(\gamma_{c,1}\gamma_{a,n}^{-1}\left(\rho_k(P)1_{\Z_p}\right))\\
&=&\hat\varphi_{\rm ext}^\pm(\gamma_{a,1}D)(\gamma_{0,n-1}^{-1}\left(\rho_k(P)1_{\Z_p}\right))\\
&=&\left(\frac{1}{p\alpha}\right)^{n-1}\cdot\varphi^\pm_{g_{n-1}\mid_{\gamma_{0,n-1}}}\left(\gamma_{a,1}D\right)\left(\gamma_{0,n-1}^{-1}P\right)\\
&=&\frac{1}{p}\left(\frac{1}{p\alpha}\right)^{n-1}\cdot\varphi^\pm_{g_{n-1}\mid_{\gamma_{0,n-1}\gamma_{a,1}}}\left(D\right)\left(\gamma_{a,1}^{-1}\gamma_{0,n-1}^{-1}P\right)\\
&=&\alpha\left(\frac{1}{p\alpha}\right)^{n}\cdot\varphi^\pm_{g_{n-1}\mid_{\gamma_{a,n}}}\left(D\right)\left(\gamma_{a,n}^{-1}P\right).
\end{eqnarray*}
Since $g_n=g_{n-1}-f_\alpha$, we deduce that 
\begin{equation}\label{actUp}
U_p\hat\varphi_{\rm ext}^\pm=\alpha\cdot\left(\hat\varphi_{\rm ext}^\pm+\hat\varphi^\pm\right),
\end{equation}
where $\hat\varphi^\pm\in \Hom^\pm_\Gamma(\Delta_0,D_k^\dagger[1])$ is the classical overconvergent modular symbol corresponding through Theorem \ref{ctrlthm} to the eigenvector with eigenvalue $\alpha$ given by $f_\alpha$.

\subsection{Specialization of $\hat\varphi_{\rm ext}^\pm$}

Theorem \ref{ctrlthm} asserts that the morphism $\rho_k^\ast$ of \eqref{morphrho} becomes an isomorphism when we restrict ourselves to generalized eigenspaces for $U_p$ with valuation of the eigenvector strictly less than $k+1$. We have seen that $\hat\varphi_{\rm ext}^\pm$ lives in the eigenspace of eigenvalue $\alpha$, and we know that $v_p(\alpha)=k/2$. Thus, it corresponds bijectively to an element of $\Hom^\pm_\Gamma(\Delta_0,V(k))$. We can easily compute the image $\rho_k^\ast \hat\varphi_{\rm ext}^\pm$ just calculating the image of the polynomical functions $\rho_k(P)1_{\Z_p}$:
\[
\hat\varphi_{\rm ext}^\pm(D)(\rho_k(P)1_{\Z_p})=\left(\frac{1}{p\alpha}\right)^0\cdot\varphi^\pm_{g_0}\left(D\right)\left(P\right)=\varphi^\pm_{f-f_\alpha}\left(D\right)\left(P\right).
\]
Thus, $\rho_k^\ast \hat\varphi_{\rm ext}^\pm=\varphi^\pm_{f-f_\alpha}$, that corresponds via Eichler-Shimura to the modular form $f-f_\alpha$. This fact fits with Theorem \ref{ctrlthm} since $f-f_\alpha$ belongs to the generalized eigenspace, indeed, $(U_p-\alpha)^2(f-f_\alpha)=0$.

\section{Extremal p-adic L-functions in families}

\subsection{Weight space}

Let $\cW/\Q_p$ be the standard one-dimensional weight space. It is a rigid analytic space that classify characters of $\Z_p^\times$, namely, 
\[
\cW=\Hom_{\rm cnt}(\Z_p^\times, \G_m).
\] 
If $L$ is any normed extension of $\Q_p$, we write $\tilde w:\Z_p^\times\rightarrow L^\times$ for the continuous morphism of groups corresponding to a point $w\in\cW(L)$.

If $k \in \Z$, then the morphism $\tilde k(t) = t^k$ for all $t \in \Z_p^\times$ defines a point in $\cW(\Q_p)$ that we will also denote by $k$. Thus $\Z\subset\cW(\Q_p)$, and we call points in $\Z$ inside $\cW(\Q_p)$ integral weights.

If $W={\rm Sp}R$ is an admissible affinoid of $\cW$, the immersion ${\rm Sp}(R)=W\hookrightarrow\cW$ defines an element $K\in\cW(R)$ such that, for every $w\in W(\Q_p)\hookrightarrow\cW(\Q_p)$, we have $\tilde w=w\circ\tilde K$. By \cite[Lemma 3.3]{Be}, there exists $r(W)>1$ such that the morphism 
\[
\Z_p\longrightarrow R^\times,\qquad z\longmapsto \tilde K(1+pz)
\]
belongs to $A[r(W)](R)$. We say that $W$ is \emph{nice} if the points $\Z\cap W$ are dense in $W$ and both $R$ and $R_0/pR_0$ are PID, where $R_0$ is the unit ball for the supremum norm in $R$.

\subsection{The Eigencurve}

For a fixed nice affinoid subdomain $W={\rm Sp}R$ of $\cW$, we can consider the $R$-modules $\Hom_\Gamma^\pm(\Delta_0,D_{\tilde K}[r])$, for $1<r\leq r(W)$. By \cite[Proposition 3.6]{Be}, we have that the space $\Hom_\Gamma^\pm(\Delta_0,D_{\tilde K}[r])$ is potentially orthonormalizable Banach $R$-module. The elements of the Hecke algebra $\cH=\Z[T_q,\langle n\rangle,U_p]$ act continuously and $U_p$ acts compactly. %Moreover, for any pair of real numbers $0<r,r'\leq r(W)$, the $R$-modules $\Hom_\Gamma^\pm(\Delta_0,D_{\tilde K}[r])$ and $\Hom_\Gamma^\pm(\Delta_0,D_{\tilde K}[r'])$ are linked in the sense of \cite{Buzz}. 

If we consider $\Hom_\Gamma^\pm(\Delta_0,D^\dagger_{\tilde K}[r])$, \cite[Theorem 3.10]{Be} asserts that, for any $w\in W(\Q_p)$ and any real number $1<r\leq r(W)$, there natural $\cH$-equivariant morphism
\begin{equation}\label{special}
\Hom_\Gamma^\pm(\Delta_0,D^\dagger_{\tilde K}[r])\otimes_{R,w}\Q_p\longrightarrow \Hom_\Gamma^\pm(\Delta_0,D_{\tilde w}[r])
\end{equation}
is always injective and surjective except when $w=0$ and the sign $\pm$ is -1.

The $R$-modules $\Hom^\pm_\Gamma(\Delta_0,D_w[r])$ for all $1<r\leq r(W)$ are all $\nu$-adapted if one is, in which case we say that $W={\rm Sp}R$ is $\nu$-adapted. If $W$ is $\nu$-adapted the restriction maps define isomorphisms between the $R$-modules $\Hom^\pm_\Gamma(\Delta_0,D_{\tilde w}[r])^{\leq \nu}$ for all $1<r\leq r(W)$.
Thus we obtain an isomorphism
\begin{equation}\label{isoOC}
\Hom^\pm_\Gamma(\Delta_0,D^\dagger_{\tilde w}[r])^{\leq\nu}\simeq\Hom^\pm_\Gamma(\Delta_0,D_{\tilde w}[r])^{\leq\nu},\qquad 1<r\leq r(W),
\end{equation}
as seen in \cite[Proposition 3.11]{Be}.

The eigencurves $\cC^{\pm}\stackrel{\kappa}{\rightarrow}\cW$ can be constructed as the union of local pieces 
\[
\cC^{\pm}_{W,\nu}\longrightarrow W={\rm Sp}R,
\]
where $\nu\in \R$ is a real and $W$ is a nice affinoid subspace adapted to $\nu$. By definition 
\[
\cC^\pm_{W,\nu}={\rm Sp}\mathbb{T}^\pm_{W,\nu},
\]
where $\mathbb{T}^\pm_{W,\nu}$ is the $R$-subalgebra of $\End_R(\Hom_\Gamma^\pm(\Delta_0,D^\dagger_{\tilde K}[1])^{\leq\nu})$ generated by the image of the Hecke algebra $\cH$.

\begin{remark}
The cuspidal parts of $\cC^+_{W,\nu}$ and $\cC_{W,\nu}^-$ coincide by \cite[Theorem 3.27]{Be}, hence we will sometimes identify certain  neighbourhoods of cuspidal points. 
\end{remark}

\subsection{Specialization}

Let $w\in W(\Q_p)$ and write $\Hom_\Gamma^\pm(\Delta_0,D^\dagger_{\tilde w}[1])^{\leq\nu}_g$ for the image of the composition. 
\begin{equation}\label{comps}
\Hom_\Gamma^\pm(\Delta_0,D^\dagger_{\tilde K}[1])^{\leq\nu}\otimes_{R,w}\Q_p\stackrel{\eqref{special}}{\longrightarrow}\Hom_\Gamma^\pm(\Delta_0,D_{\tilde w}[1])^{\leq\nu}\stackrel{\eqref{isoOC}}{\longrightarrow}\Hom_\Gamma^\pm(\Delta_0,D^\dagger_{\tilde w}[1])^{\leq\nu}
\end{equation}
In analogy with previous definition, we write $\mathbb{T}^\pm_{w,\nu}$ for the $\Q_p$-subalgebra of the endomorphism ring $\End_{\Q_p}(\Hom_\Gamma^\pm(\Delta_0,D^\dagger_{\tilde w}[1])^{\leq\nu}_g)$ generated by the image of the Hecke algebra $\cH$. By definition, there is a correspondence between points $x\in {\rm Spec}\mathbb{T}^\pm_{w,\nu}(\bar\Q_p)$ and systems of $\cH$-eigenvalues appearing in $\Hom_\Gamma^\pm(\Delta_0,D^\dagger_{\tilde w}[1])^{\leq\nu}_g$. For any such $x$, we denote by 
\[
\Hom_\Gamma^\pm(\Delta_0,D^\dagger_{\tilde w}[1])_{(x)}
\] 
the generalized eigenspace of the corresponding eigenvalues. Similarly, we denote by $(\mathbb{T}^\pm_{w,\nu})_{(x)}$ the localization of $\mathbb{T}^\pm_{w,\nu}\otimes_{\Q_p}\bar\Q_p$ at the maximal ideal corresponding to $x$. We have that
\begin{equation}\label{eqlocx}
\Hom_\Gamma^\pm(\Delta_0,D^\dagger_{\tilde w}[1])_{(x)}=\Hom_\Gamma^\pm(\Delta_0,D^\dagger_{\tilde w}[1])^{\leq\nu}\otimes_{\mathbb{T}^\pm_{w,\nu}}(\mathbb{T}^\pm_{w,\nu})_{(x)}.
\end{equation}

Since by definition $\Hom_\Gamma^\pm(\Delta_0,D^\dagger_{\tilde K}[1])^{\leq\nu}\otimes_{R,w}\Q_p\simeq\Hom_\Gamma^\pm(\Delta_0,D^\dagger_{\tilde w}[1])^{\leq\nu}_g$, we have a natural \emph{specialization map}
\[
s_w:\mathbb{T}^\pm_{W,\nu}\otimes_{R,w}\Q_p\longrightarrow \mathbb{T}^\pm_{w,\nu}.
\]
By \cite[Lemme 6.6]{Che} the morphism $s_w$ is surjective for all $w\in\cW(\Q_p)$ and its kernel is nilpotent. In particular
\[
{\rm Spec} \mathbb{T}^\pm_{w,\nu}(\bar\Q_p)=\kappa^{-1}(w)(\bar\Q_p), \qquad \kappa:\cC^\pm\longrightarrow\cW.
\] 
Given $x\in{\rm Spec} \mathbb{T}^\pm_{w,\nu}(\bar\Q_p)\subset \cC^\pm_{W,\nu}(\bar\Q_p)$, we can consider the rigid analytic localization $(\mathbb{T}^\pm_{W,\nu})_{(x)}$ of $\mathbb{T}^\pm_{W,\nu}\otimes_{\Q_p}\bar\Q_p$ at the maximal ideal corresponding to $x$.
Notice that, if we denote by $R_{(w)}$  the rigid analytic localization of $R\otimes_{\Q_p}\bar\Q_p$ at the maximal ideal corresponding to $w$, then $(\mathbb{T}^\pm_{W,\nu})_{(x)}$ is naturally a $R_{(w)}$-algebra. Localizing at $x$ we obtain a surjective local morphism of finite local $\bar\Q_p$-algebras with nilpotent kernel
\begin{equation}\label{defsw}
s_w:(\mathbb{T}^\pm_{W,\nu})_{(x)}\otimes_{R_{(w)},w}\bar\Q_p\longrightarrow (\mathbb{T}^+_{w,\nu})_{(x)}.
\end{equation}

\begin{lemma}\label{charTx}
We have that
\[
(\mathbb{T}^\pm_{w,\nu})_{(x)}\simeq \bar\Q_p[X]/X^2,
\]
where $X$ corresponds to the element of the Hecke algebra $U_p-\alpha$.
\end{lemma}
\begin{proof}
Equation \eqref{eqlocx} shows that $(\mathbb{T}^\pm_{w,\nu})_{(x)}$ is the $\Q_p$-subalgebra of the endomorphism ring $\End_{\Q_p}(\Hom_\Gamma^\pm(\Delta_0,D^\dagger_{\tilde w}[1])^{\leq\nu}_{(x)})$ generated by the image of the Hecke algebra $\cH$. By Theorem \ref{ctrlthm} we have
\[
\Hom_\Gamma^\pm(\Delta_0,D^\dagger_{\tilde w}[1])^{\leq\nu}_{(x)}=\Hom_\Gamma^\pm(\Delta_0,V(k))^{\leq\nu}_{(x)}=\bar\Q_p\hat\varphi^\pm+\bar\Q_p\hat\varphi_{\rm ext}^\pm,
\]
Hence the result follows from results of \S \ref{actionUp} and the fact that Hecke operators $T_q$ and $\langle n\rangle$ act by scalar.
\end{proof}

\begin{definition}
Any classical cuspidal non-critical $y\in\cC^\pm(\bar\Q_p)$ corresponds to a $p$-stabilized normalized cuspidal modular symbol $\varphi_{f'_{\alpha'}}^\pm$ of weight $\kappa(y)+2$. In this situation, we write
\[
\mu_y^\pm:=\mu_{f',\alpha'}^\pm.
\]
Analogously, in our irregular situation given by $x\in\cC^\pm(\bar\Q_p)$, we write
\[
\mu_{x}^{{\rm ext},\pm}:=\mu_{f,p}^{{\rm ext},\pm}.
\]
\end{definition}

\subsection{Two variable $p$-adic L-functions}\label{secLmu}

In this irregular situation, Betina and Williams define in  \cite{BW} two variable $p$-adic L-functions $\cL_p^\pm$ that interpolate the $p$-adic L-functions $\mu_y^\pm$ as $y\in\cC^\pm(\bar\Q_p)$ runs over classical points in a neighbourhood of $x\in\cC^\pm(\bar\Q_p)$. In this section, we recall their construction and we give a relation between $\cL_p^{\pm}$ and $\mu_{x}^{{\rm ext},\pm}$.

\begin{proposition}
The space $\Hom_\Gamma^\pm(\Delta_0,D^\dagger_{\tilde K}[1])_{(x)}$ is a free $(\mathbb{T}^\pm_{W,\nu})_{(x)}$-module of rank one.
\end{proposition}
\begin{proof}
\cite[Proposition 4.10]{BW}.
\end{proof}

\begin{corollary}
After possibly shrinking $W$, there exists a connected component $V={\rm Sp}(T)\subset \cC^\pm_{W,\nu}$ through $x$ such that $T$ is Gorestein and 
\[
\cM_\pm:=\Hom_\Gamma^\pm(\Delta_0,D^\dagger_{\tilde K}[1])^{\leq\nu}\otimes_{\mathbb{T}^\pm_{W,\nu}}T
\]
is a free $T$-module of rank one.
\end{corollary}
\begin{proof}
\cite[Corollary 4.11]{BW}.
\end{proof}

From the formalism of Gorestein rings, it follows that the $R$-linear dual $\cM_\pm^\vee:=\Hom_R(\cM_\pm,R)$ is free of rank one over $T$. 
Let $\cR$ be the $\Q_p$-algebra of locally analytic distributions of $\Z_p^\times$. We have a natural morphism $D^\dagger[1]\rightarrow \cR$ provided by the extension-by-zero map. This induces a morphism $\iota:D^\dagger_{\tilde K}[1]\rightarrow \cR\hat\otimes_{\Q_p}R$ and a $R$-linear morphism
\begin{eqnarray*}
{\rm Mel}:\Hom_\Gamma^\pm(\Delta_0,D^\dagger_{\tilde K}[1])&\longrightarrow &\cR\hat\otimes_{\Q_p}R\\
\varphi&\longmapsto&\iota\left(\varphi(0-\infty)\right)
\end{eqnarray*}
Since $V$ is a connected component of the eigencurve, $\cM_\pm$ is a direct summand of $\Hom_\Gamma^\pm(\Delta_0,D^\dagger_{\tilde K}[1])^{\leq\nu}$.
Thus the restriction of ${\rm Mel}$ defines an element of $\cR\hat\otimes_{\Q_p}\cM_\pm^\vee$.  
\begin{definition}
By choosing a basis of $\cM_\pm^\vee$ over $T$, the above construction provides 
\[
\cL_p^\pm\in \cR\hat\otimes_{\Q_p} T
\]
called the \emph{the two variables $p$-adic L-functions}.
\end{definition}

Write $\bar\Q_p[\varepsilon]:=\bar\Q_p[X]/(X^2)$, and let us consider the morphism
\[
x[\varepsilon]^\ast:T\longrightarrow T_{(x)}=(T_{W,\nu}^\pm)_{(x)}\longrightarrow (T_{W,\nu}^\pm)_{(x)}\otimes_{R_{(w)},w}\bar\Q_p\stackrel{s_w}{\longrightarrow}(T_{w,\nu}^\pm)_{(x)}\simeq\bar\Q_p[\varepsilon],
\]
given by \eqref{defsw} and Lemma \ref{charTx}. This provides a point $x[\epsilon]\in V(\bar\Q_p[\varepsilon])$ lying above $x\in V(\bar\Q_p)$.

\begin{theorem}
For any $y\in V(\bar\Q_p)$ corresponding to a small slope $p$-stabilized cuspidal eigenform, 
\[
\cL_p^\pm=C^\pm(y)\cdot \mu_{y}^\pm\in \cR, 
\]
for some $C\pm(y)\in \bar\Q_p^\times$. We can normalize $\cL_p^\pm$ by choosing the right $T$-basis $\phi^\pm$ of $\cM_\pm^\vee$ so that $C^\pm(x)=1$. Moreover, for a good choice of $\phi^\pm$,
\[
\cL^\pm_p(x[\epsilon])=\mu_{x}^\pm+\alpha^{-1}\mu_{x}^{{\rm ext},\pm}\varepsilon\in\cR\otimes_{\Q_p}\bar\Q_p[\varepsilon].
\]
\end{theorem}
\begin{proof}
The first part of this theorem corresponds to \cite[Theorem 5.2]{BW}. We can extend here their arguments to deduce also the second part of the theorem. 

By definition 
\[
{\rm Mel}=\cL_p^\pm\phi^\pm\in\cR\hat\otimes_{\Q_p}\cM_\pm^\vee.
\]

For any point $y\in V(\bar\Q_p)$, write $w=\kappa(y)\in W(\bar\Q_p)$. If we denote $\cM_{(y)}:=\cM_\pm\otimes_T T_{(y)}$, we have
\[
\cM_{(y)}^\vee\otimes_{R_w,w}\bar\Q_p=\Hom_{R_w}(\cM_{(y)},R_w)\otimes_{R_w,w}\bar\Q_p=\Hom_{\bar\Q_p}(\cM_{(y)}\otimes_{R_w,w}\bar\Q_p,\bar\Q_p),
\]
since $\cM_{(y)}$ is a finite free $R_w$-module. By \cite[Proposition 4.3]{BW} and the control Theorem \ref{ctrlthm}, the composition \eqref{comps} provides an isomorphism 
\begin{eqnarray*}
\cM_{(y)}\otimes_{R_w,w}\bar\Q_p&=&\Hom_\Gamma^\pm(\Delta_0,D^\dagger_{\tilde w}[1])_{(y)}\simeq\Hom_\Gamma^\pm(\Delta_0,V(w))_{(y)}\\
&=&\left\{
\begin{array}{ll}\bar\Q_p\hat\varphi_y^\pm,&\mbox{regular case,}\\\bar\Q_p\hat\varphi_y^\pm+\bar\Q_p\hat\varphi_{y,\rm ext}^\pm,&\mbox{irregular case.}\end{array}
\right.
\end{eqnarray*}
We observe that, since
\[
T_{(y)}\otimes_{R_w,w}\bar\Q_p=\left\{
\begin{array}{ll}\bar\Q_p,&\mbox{regular case,}\\\bar\Q_p[\epsilon],&\mbox{irregular case,}\end{array}
\right.
\]
a $T_{(y)}\otimes_{R_w,w}\bar\Q_p$-basis for $\cM_{(y)}^\vee\otimes_{R_w,w}\bar\Q_p$ is given by $\phi_y^\pm$ with $\phi^\pm_y(\hat\varphi_y^\pm)=1$ and $\phi^\pm_{y}(\hat\varphi_{y,{\rm ext}}^\pm)=0$. Notice first that the point $y:T\rightarrow\bar\Q_p$ factors through $T_{(y)}\otimes_{R_w,w}\bar\Q_p\rightarrow\bar\Q_p$, and fits into the commutative diagram
\[
\xymatrix{
T_{(y)}\otimes_{R_w,w}\bar\Q_p\ar[rr]^{y}\ar[d]^{\cdot\phi_y^\pm}&&\bar\Q_p\ar[d]^{=}\\
\cM^\vee_{(y)}\otimes_{R_w,w}\bar\Q_p\ar[rr]^{f\mapsto f(\hat\varphi_y^\pm)}&&\bar\Q_p
}
\]
Since $\phi_y^\pm$ corresponds to the specialization of $\phi^\pm$ up to constant,
we compute %for all $\varphi=a\hat\varphi_y^++b\hat\varphi_{y,\rm ext}^+\in \cM_{(y)}\otimes_{R_w,w}\bar\Q_p$:
\[
C^\pm(y)\cdot \mu_{y}^\pm=C^\pm(y)\cdot \hat\varphi_{y}^\pm(0-\infty)=C^\pm(y)\cdot {\rm Mel}(\hat\varphi_y^\pm)=\cL_p^\pm(y)\cdot\phi^\pm_y(\hat\varphi_y^\pm)=\cL_p^\pm(y),
\]
for some $C^\pm(y)\in\bar\Q_p$ so that $C^\pm(y)\cdot\phi^\pm=\phi^\pm_y$. This proves the first assertion. For the second, notice that $C^\pm(x)=1$ and we have the commutative diagram
\[
\xymatrix{
T_{(x)}\otimes_{R_w,w}\bar\Q_p\ar[rrrr]_{\simeq}^{x[\varepsilon]}\ar[d]^{\cdot\phi_x^\pm}&&&&\bar\Q_p[\varepsilon]\ar[d]^{=}\\
\cM^\vee_{(x)}\otimes_{R_w,w}\bar\Q_p\ar[rrrr]^{f\mapsto f(\hat\varphi_x^\pm)+\varepsilon\alpha^{-1} f(\hat\varphi_{x,{\rm ext}}^\pm)}&&&&\bar\Q_p[\varepsilon]
}
\]
since by \eqref{actUp} we have $(U_p-\alpha)\hat\varphi_{x,{\rm ext}}^\pm=\alpha\hat\varphi_x^\pm$. Again we compute
\begin{eqnarray*}
\mu_{x}^\pm+\alpha^{-1}\mu_{x}^{{\rm ext},\pm}\varepsilon&=&\hat\varphi_{x}^\pm(0-\infty)+\alpha^{-1}\hat\varphi_{x,{\rm ext}}^\pm(0-\infty)\varepsilon={\rm Mel}(\hat\varphi_x^\pm)+\varepsilon\alpha^{-1} {\rm Mel}(\hat\varphi_{x,{\rm ext}}^\pm)\\
&=&\cL_p^\pm(x[\varepsilon])\cdot\left(\phi_x^\pm(\hat\varphi_x^\pm)+\varepsilon\alpha^{-1}\phi^\pm_x(\hat\varphi_{x,{\rm ext}}^\pm)\right)=\cL_p^\pm(x[\varepsilon]),
\end{eqnarray*}
and the result follows.
\end{proof}

Notice that there is no canonical choice of $\phi_x^\pm$ even though we impose $C^\pm(x)=1$. In fact, $(1+\varepsilon c)\cdot\phi_x^\pm$ with $c\in\bar\Q_p$ is also a basis so that $C^\pm(x)=1$. For any such a change of basis we obtain
\[
\cL^\pm_p(x[\epsilon])=(1+\varepsilon c)^{-1}(\mu_{x}^\pm+\alpha^{-1}\mu_{x}^{{\rm ext},\pm}\varepsilon)=\mu_{x}^\pm+(\alpha^{-1}\mu_{x}^{{\rm ext},\pm}-c\mu_x^{\pm})\varepsilon.
\]
The following result does not depend on the choice of the generator $\phi^\pm$:
\begin{corollary}
Let $t\in T$ the element corresponding to $U_p-\alpha$. Then
\[
\frac{\partial \cL_p^\pm}{\partial t}(x)\in  \alpha^{-1}\mu_{x}^{{\rm ext},\pm}+\bar\Q_p\mu_x^{\pm}.
\]
\end{corollary}

\bibliographystyle{plain}
\bibliography{biblio}
\end{document}